\documentclass[reqno]{amsart}
\usepackage{txfonts}
\usepackage{amsfonts}
\usepackage{amssymb}
\usepackage[mathscr]{eucal}
\usepackage{amsmath}
\usepackage{array,longtable}
\usepackage[bookmarksopen,colorlinks,citecolor=blue,
linkcolor=black,pdfstartview=FitH]{hyperref}
\usepackage{float}
\usepackage{url}
\usepackage{appendix}
\usepackage{graphicx}
\usepackage{enumitem}
\usepackage{algorithm}
\usepackage{algpseudocode}

\newtheorem{theorem}{Theorem}[section]
\newtheorem{corollary}[theorem]{Corollary}
\newtheorem{lemma}[theorem]{Lemma}
\newtheorem{proposition}[theorem]{Proposition}
\theoremstyle{definition}
\newtheorem{definition}[theorem]{Definition}
\newtheorem{notation}[theorem]{Notation}

\numberwithin{equation}{section}
\theoremstyle{remark}
\newtheorem{remark}[theorem]{Remark}

\newcounter{condition}

\allowdisplaybreaks[4]
\makeatletter
\@namedef{subjclassname@2020}{\textup{2020} MSC}
\makeatother

\begin{document}

\title{Substitution and quotient of the isotropy group action}
\pdfbookmark[0]{ }{}

\author[li]{Xin Li}
\address{School of Mathematical Sciences, Zhejiang University of Technology, Hangzhou 310023, Zhejiang, P. R. China}
\email{xinli1019@126.com} 

\author[Wang]{Yu Wang}
\address{College of Science, National University of Defense Technology, Changsha 410072, Hunan, China.}
\email{wangyu25@nudt.edu.cn}

\author[Hu]{Shenglong Hu}
\address{College of Science, National University of Defense Technology, Changsha 410072, Hunan, China.}
\email{hushenglong@nudt.edu.cn}

\thanks{Xin Li's research is supported by National Natural Science Foundation of China (Grant No. 11801506). The research of Yu Wang and Shenglong Hu is partially supported by the Natural Science Foundation of Hunan Province of China (Grant No. 2025JJ20006), the National Science Foundation of China (Grant No. 12571334), and the Innovation Research Foundation of National University of Defense Technology.}


\begin{abstract}
Given a solution of the Brent equations, a partial solution can be substituted into the Brent equations which makes it a reduced polynomial system. Then the reduced polynomial system can be solved by symbolic software, and a parameterized solution set can often be found. However, the solution set admits a positive-dimensional isotropy group action. The parameterized solution set may lie entirely in a single isotropy group orbit. In this paper, we give a geometric explanation of making the substitution. Considering the foliation structure of the solution set, a method on choosing the partial solution is proposed. By the proposed method, parameterized solution sets that intersect distinct isotropy group orbits can be efficiently determined. In particular, the rational solution found by Dumas, Pernet and Sedoglavic is parameterized. In the parameterized solution set, we can find infinitely many inequivalent 48-multiplication algorithms with only rational coefficients.
\end{abstract}
\maketitle

\section{Introduction}\label{se:intro}
\

In the seminal paper \cite{Stra-69}, Strassen discovered an algorithm that computes the multiplication of two $2\times 2$ matrices by 7 multiplications instead of 8. This method has been extended and developed into new classes of algorithms for matrix multiplication, known as Strassen-type algorithms. It is well-known that Strassen-type algorithms can be obtained by solving polynomial systems. More precisely, if we want to find an algorithm which computes the multiplication of an $m\times n$ matrix and an $n\times p$ matrix using $r$ multiplications, we can solve a polynomial system denoted by $B(m,n,p|r)$, which is called \emph{the Brent equations} \cite{Brent70,Heule21}.

Let $\mathrm{V}(m,n,p|r)$ denote the solution set of $B(m,n,p|r)$. From de Groote’s discussion \cite{deGr78-1,deGr78-2}, we know that $\mathrm{V}(m,n,p|r)$ is positive-dimensional and admits a positive-dimensional isotropy group action. Let $\mathcal{I}(m,n,p|r)$ denote the corresponding isotropy group. It is well-known that $\mathcal{I}(m,n,p|r)$ is generated by four types of groups \cite{Bur15,Heule21,kaumoo22}:
\begin{enumerate}
  \item the de Groote group;
  \item the layer scaling group;
  \item the group of permutation summands;
  \item the group of $S_3$-symmetry.
\end{enumerate}
The  de Groote and layer scaling groups are positive-dimensional Lie groups. The group of permutation summands and  $S_3$-symmetry are finite discrete groups.

In \cite{deGr78-2}, de Groote showed that $\mathrm{V}(2,2,2|7)$ consists of a single isotropy group orbit.  Then the dimension of $\mathrm{V}(2,2,2|7)$ is equal to the dimension of the  $\mathcal{I}(2,2,2|7)$-orbit. However, if the dimension of $\mathrm{V}(m,n,p|r)$ is greater than the dimension of $\mathcal{I}(m,n,p|r)$-orbits, there are infinitely many different $\mathcal{I}(m,n,p|r)$-orbits in $\mathrm{V}(m,n,p|r)$. 

Since the polynomial system $B(m,n,p|r)$ is large-scale, 
computing the full solution set is currently infeasible \cite{Heule21}. Nevertheless, if we find one new solution, then we can try to search other new solutions in its neighbourhood by fixing a partial solution and making the substitution. In \cite{JohnMc86}, by fixing partial solutions, Johnson and McLoughlin found two parameterized solution sets in $\mathrm{V}(3,3,3|23)$, and showed that one of them contains infinitely many points that belong to different $\mathcal{I}(3,3,3|23)$-orbits.  After that, in \cite{Heule21} Heule et al. found a subset of $\mathrm{V}(3,3,3|23)$ with 17 parameters. Other parameterized solution sets can be found in \cite{CV25}.

However, there are some obstacles when fixing the partial solutions. First, after fixing the partial solution, the known solution may become an isolated solution in the reduced polynomial system. Second, even if one parameterized solution set is obtained, it may be contained in a single orbit, and in this case, the parameterized solution set obtained is not new. In this paper, we propose a method that can overcome these obstacles. In many cases, the new parameterized solution set can be computed efficiently.

To that end, tools from Lie group theory \cite{olver93li,olver95e,olver99c,olver03mf} are employed. Since $\mathcal{I}(m,n,p|r)$ is a positive-dimensional Lie group, a simple but key observation is that the group action of $\mathcal{I}(m,n,p|r)$ on $\mathrm{V}(m,n,p|r)$ makes its smooth part a foliation whose leaves are the group orbits \cite{lawson74fol,lee12smooth}. Fixing a partial solution and making the substitution can be viewed as the intersection between an affine coordinate subspace and $\mathrm{V}(m,n,p|r)$. Then if a suitable affine coordinate subspace is chosen, after intersection the parameterized solution set contains points in distinct group orbits. To find such an affine coordinate subspace, we propose a rank condition. For the Brent equations, we give a practical substitution procedure based on the rank condition. Numerical tests for solutions in $\mathrm{V}(3,3,3|23)$, $\mathrm{V}(4,4,4|48)$ and $\mathrm{V}(4,4,4|49)$ are carried out, and many known solutions can be efficiently parameterized. In particular, the rational solution in $\mathrm{V}(4,4,4|48)$ found by Dumas, Pernet and Sedoglavic is parameterized \cite{dumas25}. A one-dimensional rational parameterized curve is obtained, which contains infinitely many rational points that belong to distinct $\mathcal{I}(4,4,4|48)$-orbits. The result is given in Appendix \ref{sec-app}.

The paper is organized as follows. In Section \ref{sec-subnull}, the preliminaries and basic results for making the substitution are introduced. In Section \ref{sec-subslg}, we discuss how to fix partial solutions when the solution set of the polynomial system has an isotropy group action. Based on the tangent and nullspace basis matrices, a rank condition is proposed in Section \ref{subsec-ntmethod}, which can help us find available partial solutions to be fixed. In Section \ref{sec-tnsbe}, we give explicit formulas for calculating the tangent and nullspace basis matrices for solutions of Brent equations. In Section \ref{sec-numtest}, we propose a practical substitution procedure for the Brent equations and present the corresponding computational experiment results. In Section \ref{sec-remark}, some remarks and open problems are given.

The computations are performed in Julia version~1.12.2
\cite{Bez17Julia} on a laptop equipped with an Intel Core Ultra
7 155H processor and 32.0~GB of RAM. The source code is available in
\cite{L-23}.

\section{Making the substitution by the nullspace}\label{sec-subnull}

In this section, we give some elementary results on making the substitution for polynomial systems. First, some preliminaries are given.

\subsection{Preliminaries and notation}
\

Throughout this paper, we focus on the real polynomial systems and their real solution set. In this section, we introduce some notation and properties that will be used in the following sections.

We let $[n]$ denote the set $\{1,2,\ldots,n\}$. 
If $I=\{i_1,i_2,\ldots,i_k\}$ is a subset of $[n]$,  we let $|I|$ denote the number of elements of $I$, so $|I|=k$. The set of all $n\times m$ real matrices is denoted by $\mathbb{R}^{n\times m}$. If $A\in \mathbb{R}^{n\times m}$ is a matrix, let $A_I\in \mathbb{R}^{k\times m}$ denote the submatrix of $A$ obtained by choosing the rows indexed by $I$. The transpose of $A$ is denoted by $A^t$.
For a vector $v=(v_1,v_2,\ldots,v_n)^t\in \mathbb{R}^n$, we let $v_I=(v_{i_1},v_{i_2},\ldots,v_{i_k})^t$ denote the restriction of $v$ to $I$.  For a set of vectors $\{v_i|i=1,2,\ldots, k\}$ in $\mathbb{R}^n$, we let $\langle v_1, v_2,\ldots, v_k \rangle \subseteq \mathbb{R}^n$ denote the linear subspace spanned by them. Now, we give some lemmas.
\begin{lemma}\label{lem-rankaai}
Let $A=(a_1,a_2,\ldots,a_m)\in \mathbb{R}^{n\times m}$. Let $V_A=\langle a_1,a_2,\ldots a_m \rangle$ be the linear subspace of $\mathbb{R}^{n}$ spanned by $a_i$. Let $V_A^I=\{v\in V_A| v_I=0 \}$.
Then $V_A^I \subseteq V_A$ is a linear subspace with dimension
$$\dim V_A^I=\operatorname{rank}A-\operatorname{rank}A_I.$$

In particular, if $\operatorname{rank}A=\operatorname{rank}A_I$, then $\dim V_A^I=0$.
\end{lemma}

\begin{proof}
Consider the coordinate restriction map
$$\pi_I:V_A\longrightarrow \mathbb{R}^{|I|},\qquad \pi_I(v)=v_I,$$
where $v_I$ denotes the vector consisting of the coordinates of $v$
indexed by $I$. Then $\pi_I$ is linear. Moreover,
$$\ker \pi_I=V_A^I.$$
Hence $V_A^I$ is a linear subspace of $V_A$.

Since every $v\in V_A$ can be written as $v=A c$
for some $c\in\mathbb{R}^m$, we have
$$v_I=A_I c.$$
Therefore,
$$\operatorname{Im}\pi_I=\{A_I c \mid c\in\mathbb{R}^m\},$$
which is exactly the column space of $A_I$. Hence
$$\dim \operatorname{Im}\pi_I=\operatorname{rank} A_I.$$
By the rank-nullity theorem (see, e.g., \cite[Thm 4.1.6]{artin11}),
$$\dim V_A=\dim \ker\pi_I+\dim\operatorname{Im}\pi_I.$$
Since $\dim V_A=\operatorname{rank} A$, we obtain
$$\dim V_A^I=\dim\ker\pi_I=\operatorname{rank} A-\operatorname{rank} A_I.$$
\end{proof}

\begin{lemma}\label{lem-subfull}
Suppose that $A=(a_1,a_2,\ldots,a_m)\in \mathbb{R}^{n\times m}$  and  $B=(b_1,b_2,\ldots,b_k)\in \mathbb{R}^{n\times k}$ where $a_i,~b_i\in \mathbb{R}^{n}$. Let $A_I$ and $B_I$ be the row submatrices of $A$ and $B$
indexed by $I$. Let $\langle a_1,a_2,\ldots,a_m \rangle$ and $\langle b_1,b_2,\ldots,b_k \rangle$ be linear subspaces of $\mathbb{R}^{n}$ spanned by $a_i$ and $b_i$. Let $V_A=\langle a_1,a_2,\ldots,a_m \rangle$ and 
 $V_B=\langle b_1,b_2,\ldots,b_k \rangle$. Suppose that $V_B\subseteq V_A$. Given a subset $I\subseteq [n]$, we have
$$\operatorname{rank}A_I\geq \operatorname{rank}B_I$$
and
$$\operatorname{rank}A-\operatorname{rank}A_I\geq \operatorname{rank}B-\operatorname{rank}B_I.$$
\end{lemma}
\begin{proof}
Let $a_{i,I}$ $(1\leq i\leq m)$ denote the restriction of $a_i$ to $I$ and similarly, for $b_{j,I}$ $(1\leq j\leq k)$. Then we have $A_I=(a_{1,I}, a_{2,I},\ldots,a_{m,I})$ and $B_I=(b_{1,I}, b_{2,I},\ldots,b_{k,I})$.
Since $\langle b_1,b_2,\ldots,b_k \rangle\subseteq \langle a_1,a_2,\ldots,a_m \rangle$, we have $\langle b_{1,I},b_{2,I},\ldots,
b_{k,I} \rangle\subseteq \langle a_{1,I},a_{2,I},\ldots,a_{m,I} \rangle$.
So we have
$$\operatorname{rank}B_I=\dim \langle b_{1,I},b_{2,I},\ldots,
b_{k,I} \rangle 
\leq \dim \langle a_{1,I},a_{2,I},\ldots,a_{m,I} \rangle=\operatorname{rank}A_I.$$

Let $V_A^I=\{v\in V_A| v_I=0 \}$ and $V_B^I=\{v\in V_B| v_I=0 \}$. Then we also have that $V_B^I\subseteq V_A^I$. By Lemma \ref{lem-rankaai}, we have $$ \operatorname{rank}A-\operatorname{rank}A_I=\dim V_A^I \geq \dim V_B^I=\operatorname{rank}B-\operatorname{rank}B_I.$$
\end{proof}
Then by Lemma \ref{lem-subfull} we have the following corollary.
\begin{corollary}\label{cor-subfull}
In Lemma \ref{lem-subfull},
if $\langle a_1,a_2,\ldots,a_m \rangle=\langle b_1,b_2,\ldots,b_k \rangle$ then for any subset $I\subseteq [n]$ we have
$$\operatorname{rank}A_I=\operatorname{rank}B_I$$
and  
$$\operatorname{rank}A-\operatorname{rank}A_I= \operatorname{rank}B-\operatorname{rank}B_I.$$
\end{corollary}

\subsection{Substitution by fixing a partial solution}
\

Let $F: \mathbb{R}^n\to \mathbb{R}^m$ be a polynomial mapping defined by 
$$F(x)=
\left(
\begin{array}{c}
f_1(x)\\
f_2(x)\\
\vdots\\
f_m(x)
\end{array}
\right),
$$
where $x=(x_1,x_2,\ldots,x_n)$. The corresponding system of polynomial equations $F(x)=0$ is given by
\begin{equation}\label{eq-FPE}
f_i(x_1,x_2,\ldots,x_n)=0\quad\text{where}~i=1,2,\ldots,m.
\end{equation}
The Jacobian matrix of $F$ is given by 
\begin{equation}\label{eq-jac}
J(x) =\left(\frac{\partial f_i}{\partial x_j}\right)_{i,j},~\text{where}~1\leq i\leq m,~1\leq j\leq n.
\end{equation}
Let $V(F)=\{x\in \mathbb{R}^n|F(x)=0\}$ be the solution set of (\ref{eq-FPE}). Given $s\in V(F)$, we let $\dim_s V(F)$ denote the local dimension of $s$ \cite{cox25,li25def}.  
\begin{definition}\cite{cox25}
A solution $s\in V(F)$ is called \emph{smooth} (or nonsingular) if
$\dim_s V(F)=d$, where $d=n-\operatorname{rank}J(s)$.
\end{definition}
So if $s$ is a smooth solution whose local dimension is $d$, then 
$\operatorname{rank}J(s)=n-d$. The nullspace of $J(s)$ is defined by
$$\mathrm{Null}(s)=\{x\in \mathbb{R}^n| J(s)x=0\},$$
which is a $d$-dimensional subspace of $\mathbb{R}^n$.
Choosing a (typical) basis $\{v_1,v_2,\ldots,v_d\} \subseteq~\mathrm{Null}(s)$, we define the nullspace basis matrix of $s$ by
\begin{equation}\label{eq-ns}
N(s)=(v_1,v_2,\ldots,v_d)\in \mathbb{R}^{n\times d}.
\end{equation}
Then we have $\operatorname{rank}N(s)=d$. Let $I\subseteq [n]$ be a subset and $|I|=k$. For a solution $s=(s_1,s_2,\ldots,s_n)$, we let 
$$s_I=(s_{i_1},s_{i_2},\ldots,s_{i_k}),$$
where $i_j\in I$ and $i_1<i_2<\cdots<i_k$. Then $s_I$ is called \emph{a partial solution} of $F(x)=0$ (see, e.g., \cite{cox25}). By fixing the partial solution and adjoining the constraints $x_i-s_i=0$ for $i\in I$ to $F(x)=0$, we obtain the constrained polynomial system $F_I(x)=0$:
\begin{equation}\label{eq-FI}
\left\{
\begin{array}{r@{\;}c@{\;}r@{\;}c@{\;}l}
\multicolumn{3}{r}{F(x)} & = & 0,\\
x_{i_1} & - & s_{i_1} & = & 0,\\
x_{i_2} & - & s_{i_2} & = & 0,\\
        & \vdots &     &   &   \\
x_{i_k} & - & s_{i_k} & = & 0.
\end{array}
\right.
\end{equation}
Then the corresponding polynomial mapping is:
$$F_I(x)=
\left(
\begin{array}{c}
F(x)\\
x_{i_1}-s_{i_1}\\
\vdots\\
x_{i_k}-s_{i_k}
\end{array}
\right),
$$
where $x=(x_1,x_2,\ldots,x_n)$.
Clearly, $s$ is also a solution of $F_I(x)=0$.

\begin{remark}
To solve the constrained system $F_I(x)=0$ in \eqref{eq-FI}, we
may substitute $x_i=s_i$ for all $i\in I$ into $F(x)$. This yields
a reduced polynomial system in the variables $x_j$ where $j\in[n]\setminus I$. Without of confusion, we also denote this reduced system by
$F_I(x)=0$. We refer to this procedure as \emph{making a
substitution}, or simply \emph{substitution}.
\end{remark}

Let 
\begin{equation}\label{eq-VFI}
V(F_I)=\{x\in \mathbb{R}^n|F_I(x)=0\}
\end{equation}
denote the solution set of $F_I(x)=0$. Then we have $V(F_I)\subseteq V(F)$. 

\begin{remark}
Let 
$$L_I=\{x\in \mathbb{R}^n|x_i=s_i~\text{for}~i\in I\}$$
be an affine coordinate subspace. 
So making a substitution can be understood as intersecting $V(F)$ with $L_I$. Thus, we have $V(F_I)=V(F)\cap L_I$.
\end{remark}

\begin{definition}
Let  $s\in V(F)$ be a smooth solution. A partial solution $s_I$ is called 
\emph{regular} if $s$ is a smooth solution of $F_I(x)=0$. The corresponding subset $I\subseteq [n]$ is called \emph{a regular coordinate set}.
\end{definition}
\begin{remark}
It is not hard to find a partial solution that is not regular. For example,
consider $f(x,y,z)=xy-z=0$ in $\mathbb{R}^3$, $s=(0,0,0)$ is a smooth solution. However, since (0, 0) is not a smooth point on $xy=0$, the partial solution $s_3=0$ is not regular.
\end{remark}


Let $e_i = (0,\ldots,0,1,0,\ldots,0) \in \mathbb{R}^{1\times n}$ be the $i$-th standard basis row vector, whose $i$-th entry is $1$ and all other entries are zero. For $I=\{i_1,i_2,\ldots,i_k\}$, let
$$E_I=\left(
\begin{array}{c}
e_{i_1}\\
e_{i_2}\\
\vdots\\
e_{i_k}
\end{array}
\right).
$$
Let $J_I(x)$ be the Jacobian matrix of $F_I(x)$. It is not hard to see that 
\begin{equation}\label{eq-JI}
 J_I(x)=\left(
\begin{array}{c}
J(x)\\
E_I
\end{array}
\right).
\end{equation}

Let $\mathrm{Null}_I(x)$ denote the nullspace of $J_I(x)$. Let 
\begin{equation}\label{eq-nis}
N_I(s) 
\end{equation}
be the submatrix of $N(s)$ obtained by choosing rows indexed by $I$, and denote $d_{N_I}=\operatorname{rank}N_I(s)$.

\begin{proposition}\label{prop-nbm}
Let $s_I$ be a partial solution.
The nullspace of $J_I(s)$ is denoted by $\mathrm{Null}_I(s)$.
Then we have
$$\mathrm{Null}_I(s)=\{v|v\in \mathrm{Null}(s)~\text{and}~v_I=0\}$$
and
$$\dim \mathrm{Null}_I(s)=\operatorname{rank}N(s)-\operatorname{rank}N_I(s)=d-d_{N_I}.$$
\end{proposition}
\begin{proof}
By definition,   we have
$$\mathrm{Null}_I(s)=\{v\in \mathbb{R}^n| J_I(s)v=0\}.$$
From (\ref{eq-JI}), we can see that
if $J_I(s)v=0$, then $J(s)v=0$. So the nullspace of $J_I(s)$ is a linear subspace of the nullspace of $J(s)$.
On the other hand, if $J_I(s)v=0$, then $E_I v=0$ which implies $v_I=0$.
Conversely, if $v\in \mathrm{Null}(s)$ and $v_I=0$, then $v\in \mathrm{Null}_I(s)$.
So we have 
$$\mathrm{Null}_I(s)=\{v|v\in \mathrm{Null}(s)~\text{and}~v_I=0\}.$$

Then by Lemma \ref{lem-rankaai}, the dimension of $\mathrm{Null}_I(s)$ is
$$\dim \mathrm{Null}_I(s)=\operatorname{rank}N(s)-\operatorname{rank}N_I(s)=d-d_{N_I}.$$
\end{proof}
By Proposition \ref{prop-nbm}, we have the following corollary.
\begin{corollary}\label{cor-ddi}
Suppose that $I$ is a regular coordinate set.
If $d-d_{N_I}=0$, then $s$ is an isolated solution of $F_I(x)=0$.
If $d-d_{N_I}>0$, then $s$ is a positive-dimensional solution of $F_I(x)=0$ with local dimension $d-d_{N_I}$.
\end{corollary}

\subsection{Making the substitution by nullspace basis matrix} 
\label{subsec-first}
\

Let $F(x)=0$ be a (large) polynomial system whose solution set $V(F)$ is positive-dimensional. Assume that $s$ is a smooth point of a positive-dimensional component of $V(F)$. After making the substitution as in (\ref{eq-FI}), the original system is reduced to a smaller system $F_I(x)=0$ with solution set $V(F_I)$. If we choose $I$ such that $|I|$ is large enough, then $F_I(x)=0$ has lower complexity. We can solve it by symbolic computation software, such as, Macaulay2, Oscar.jl and Singular \cite{Dec25,Eis-M2book02,GreuelPf08}. However, after making a substitution, $s$ may become an isolated solution of $F_I(x)=0$. By Corollary \ref{cor-ddi}, we obtain an elementary first-order necessary condition for making the substitution, which can keep $s$ to be positive-dimensional in $V(F_I)$. That is, we can use the partial solution $s_I$ indexed by $I\subseteq [n]$ such that $$\operatorname{rank}N(s)-\operatorname{rank}N_I(s)>0.$$

However, when the solution set is invariant under a positive-dimensional Lie group action, such as the Brent equations\cite{Heule21,li25def}, all solutions of the reduced system may belong to a single group orbit. So in this case, the solutions we found are not new up to the group action. In the next section, we will discuss how to choose $I$ such that $V(F_I)$ contains solutions in different group orbits.

\section{Substitution modulo a Lie group action}
\label{sec-subslg}
In this section, we discuss how to choose a partial solution such
that the solution set of the reduced polynomial system contains
points from distinct group orbits. In the following discussion, the
manifold $M$ is assumed to be the smooth part of a real affine algebraic
variety in $\mathbb{R}^n$. We assume that $G$ acts linearly on the ambient space $\mathbb{R}^n$, that $M$ is $G$-invariant, and that the action of
$G$ on $M$ is obtained by restriction.
\subsection{Group orbit and tangent basis matrix} 
\

Suppose that  $M\subseteq \mathbb{R}^n$ is an $m$-dimensional manifold. Let $G$ be a $d$-dimensional Lie group that acts on $M$. Let  $\mathfrak{g}$ be the Lie algebra of $G$. Therefore, as a linear space we have $\dim \mathfrak{g}=d$. The \emph{infinitesimal generators} of the group action are given by \cite[(1.48)]{olver93li}:
\begin{equation}\label{eq-infgen}
A\cdot x =\left.\frac{d}{dt}\exp(tA)x\right|_{t=0},
\quad x \in \mathbb{R}^n,\; A \in \mathfrak{g}.
\end{equation}
So we have $A\cdot x\in \mathbb{R}^n$.

Let $\{E_1,E_2,\ldots,E_k\}\subseteq \mathfrak{g}$ be a subset.
Let $\langle E_1,E_2,\ldots,E_k\rangle$ be the linear space spanned by $\{E_1,E_2,\ldots,E_k\}$.
Let $s\in M$ be a point. Let $Gs$ be the group orbit of $s$, which is a submanifold of $M$. 
\emph{The stabilizer group} $\operatorname{Stab}_G(s)$  of $s\in M$ is defined as the subgroup that fixes $s$:
$$\operatorname{Stab}_G(s)=\{g\in G|gs=s\}.$$
The group action is \emph{semi-regular} if $\dim \operatorname{Stab}_G(s)$ is constant for all $s\in M$. The group $G$ acts \emph{regularly} if the action is semi-regular, and, in addition, for each point $s\in M$ there exist arbitrarily small neighbourhoods $U$ of $s$  with the property that each orbit of $G$ intersects $U$ in a pathwise connected subset \cite{olver93li}.

Suppose that $\langle E_1,E_2,\ldots,E_k\rangle=\mathfrak{g}$. Let $\langle E_1\cdot s, E_2\cdot s,\ldots,E_k\cdot s\rangle \subseteq \mathbb{R}^n$ denote the linear space spanned by the infinitesimal generators $\{ E_i\cdot s|i=1,2,\ldots,k\}$. Let $T_s(Gs)$ be the \emph{tangent space} of $s$ for the submanifold $Gs$. We have the following proposition.
\begin{proposition}\cite[Prop. 2.65]{olver95e}\label{prop-tdim}
 For the submanifold $Gs$, the tangent space of $s$ is given by:
$$T_s (Gs)=\mathfrak{g}\cdot s=\langle E_1\cdot s, E_2\cdot s,\ldots,E_k\cdot s\rangle.$$
For the dimension of the orbit, we have
\begin{equation*}
\dim Gs=\dim \langle E_1\cdot s, E_2\cdot s,\ldots,E_k\cdot s\rangle.
\end{equation*}
\end{proposition}

By Proposition \ref{prop-tdim}, with the notations above we have the following definition.
\begin{definition}\label{def-tsm}
The \emph{tangent basis matrix} (or infinitesimal generator coefficient matrix \cite{folver99mcf, Hubert07sa}) of $s$ with respect to the Lie group $G$ is defined as
\begin{equation}\label{eq-Ts}
T(s)=\left(E_1\cdot s, E_2\cdot s,\cdots,E_k\cdot s\right)
\in \mathbb{R}^{n\times k}.
\end{equation}
Given a subset $I\subseteq [n]$, let 
\begin{equation}\label{eq-Tis}
T_I(s)
\end{equation}
the submatrix of $T(s)$ obtained by selecting
the rows indexed by $I$. So if $|I|=r$, we have $T_I(s)\in \mathbb{R}^{r\times k}$.
\end{definition}

The following proposition gives the connection between the dimension of the stabilizer at a point and the dimension of the orbit through that point.
\begin{proposition}\cite[Prop.2.67]{olver95e}\label{prop-dimstab}
If $G$ is an $r$-dimensional Lie group acting on $M$, then the stabilizer $\operatorname{Stab}_G(x)$ of any point $x \in M$ has dimension $r-d$, where $d$ is the dimension of the orbit of $G$ through $x$.
\end{proposition}

From Proposition \ref{prop-dimstab}, we have the following corollary.
\begin{corollary}\label{cor-dimos}
Let $T(s)$ be the tangent basis matrix defined in (\ref{eq-Ts}) for a point $s\in M$. The dimension of the group orbit is given by 
$$\dim Gs= \operatorname{rank}T(s).$$
The dimension of the stabilizer is given by
$$\dim \operatorname{Stab}_G(s)=\dim \mathfrak{g}-\operatorname{rank}T(s).$$
In particular, an action is semi-regular if all the orbits have the same dimension.
\end{corollary}

\begin{definition}\label{def-tranint}
Let $K$, $L$ be two submanifolds of $M$. Suppose that the intersection $K\cap L$ is also a smooth manifold. For a point $s\in K\cap L$, we say $K\cap L$ is \emph{a local transverse intersection at} $s$, if
 $$T_s (K\cap L)=T_s K\cap T_s L.$$
If $T_s (K\cap L)=T_s K\cap T_s L$ for all $s\in K\cap L$, then we say $K\cap L$ is \emph{a transverse intersection}. 
\end{definition}
Here our definition of transverse intersection is weaker than \cite{lee12smooth}, which is called \emph{clean intersection} in \cite[Append. C.3]{Hor85III}.

Let $M\subseteq \mathbb{R}^n$ be a smooth $m$-dimensional manifold. Suppose that there is a Lie group $G$ that acts on $M$ regularly, and $\dim M>\dim G$. Then the connected components of the group orbits form a
partition of $M$. For a point $x\in M$, since $\dim M>\dim Gx$, every open neighbourhood of $x$ in $M$ intersects infinitely many group orbits.
Following the definition in \cite{lawson74fol,lee12smooth}, $M$ is endowed with \emph{a foliation}  whose \emph{leaves} consist of the connected group orbits $Gx$. That is, there exists a set of representatives $\mathcal{A}\subset M$ such that
\begin{equation*}
M=\bigsqcup_{x\in \mathcal{A}} Gx.
\end{equation*}

Let $s\in M$ be a point. Since $M\subseteq \mathbb{R}^n$, we let $s=(s_1,s_2,\ldots,s_n)$ be the coordinate of $s$. For an index set $I\subseteq [n]$, let 
\begin{equation}\label{eq-LI}
L_I=\{x \in \mathbb{R}^n\mid x_i=s_i~\text{for all}~i\in I\}  
\end{equation}
be an affine coordinate subspace. Consider the intersection 
\begin{equation}\label{eq-MI}
M_I=M\cap L_I,
\end{equation}
then we have
\begin{equation}\label{eq-sliceo}
M_I=\bigsqcup_{x\in \mathcal{A}} (Gx\cap L_I).
\end{equation}
Since $s\in Gs\cap L_I$, we have $Gs\cap L_I\neq \emptyset$ and $M_I\neq \emptyset$. In (\ref{eq-sliceo}), each connected component of $Gx\cap L_I$ is called \emph{an orbit slice}. From Definitions \ref{def-tsm} and \ref{def-tranint}, we have the following proposition.
\begin{proposition}\label{prop-dim}
Suppose that the action of $G$ on $M$ is regular.
Assume that $Gs\cap L_I$ is a connected smooth submanifold. Suppose that
$Gs$ and $L_I$ intersect transversely along $Gs\cap L_I$.
Let  $\dim_x Gs\cap L_I$ be the local dimension of $Gs\cap L_I$ at $x$. Then we have 
$$\dim_x Gs\cap L_I=\dim_s Gs\cap L_I=\operatorname{rank}T(s)-\operatorname{rank}T_I(s),\quad\text{for all}\quad x\in Gs\cap L_I.$$
The tangent space $T_s (Gs\cap L_I)$ consists of 
$$\{v|v\in \langle E_1\cdot s,\ldots, E_k \cdot s \rangle~\text{and}~ v_I=0\}.$$

In particular, if $\operatorname{rank}T_I(s)=\operatorname{rank}T(s)$, then $\dim Gs\cap L_I=0$, which implies $s$ is an isolated point of $Gs\cap L_I$.
\end{proposition}

\begin{proof}
For a smooth manifold, the local dimension is a constant which is the dimension of the manifold. That is, 
$$\dim_x Gs\cap L_I=\dim_s Gs\cap L_I,\quad\text{for all}\quad x\in Gs\cap L_I.$$
In the following, we will show
$$\dim_s Gs\cap L_I=\operatorname{rank}T(s)-\operatorname{rank}T_I(s).$$

Since $Gs\cap L_I$ is smooth and local transverse at $s$, we have 
$$\dim_s Gs\cap L_I=\dim T_s (Gs\cap L_I),$$
and
$$T_s (Gs\cap L_I)=T_s Gs \cap T_s L_I.$$

The tangent space of $L_I$ at $s$ is
$$T_s L_I=\{x\in \mathbb{R}^n|x_i=0~\text{for}~i\in I\}.$$
Suppose that $\mathfrak{g}=\langle E_1, E_2,\ldots, E_k\rangle$. Then we have
\begin{align*}
 T_s (Gs\cap L_I)=&T_s (Gs) \cap T_s L_I \\
      =&\langle E_1\cdot s,\ldots, E_k \cdot s \rangle \cap \{x\in \mathbb{R}^n|x_i=0~\text{for}~i\in I\}
\end{align*}
So $v\in T_s (Gs\cap L_I)$ if and only if $v\in \langle E_1\cdot s,\ldots, E_k \cdot s \rangle$ and $v_I=0$. Thus, by Lemma \ref{lem-rankaai} we have
$$\dim T_s (Gs\cap L_I)=\operatorname{rank}T(s)-\operatorname{rank}T_I(s).$$

So if $\operatorname{rank}T(s)-\operatorname{rank}T_I(s)=0$, then $\dim_s Gs\cap L_I=0$, which implies $s$ is isolated.
\end{proof}

\subsection{Two typical intersection scenarios and  regular coordinate intersection}\label{subsec-two}
\

After intersection, in (\ref{eq-sliceo}) there are two typical scenarios that happen on $M_I$. We list them as follows.

\textbf{Condition A:}
The first case is that $M_I$ is also a foliation whose leaves consist of lower dimensional orbit slices.
More precisely, let
$$\mathcal A_I=\{x\in\mathcal A\mid Gx\cap L_I\neq \emptyset\}.$$
For each $x\in\mathcal A_I$, assume that each connected component of
$Gx\cap L_I$ is a smooth submanifold of $M_I$, and that $Gx$ and $L_I$ intersect transversely along $Gx\cap L_I$. Suppose that these orbit slices are connected and have the same dimension. Then the sets $Gx\cap L_I$ form leaves of the foliation on $M_I$.
Then by Proposition \ref{prop-dim}, we have
$$\dim(Gx\cap L_I)=\dim(Gs\cap L_I)=\operatorname{rank}T(s)-\operatorname{rank}T_I(s),
\quad\text{for all}\quad x\in \mathcal{A}_I.$$
So if 
$$\dim M_I>\dim(Gs\cap L_I),$$
then an open neighbourhood of $s$ in
$M_I$ intersects infinitely many distinct group orbits.

\textbf{Condition B:} The second case is $M_I=Gs\cap L_I$. So in this case, $M_I$ is just an orbit slice. Thus, we have $\dim M_I=\dim Gs\cap L_I$.

Following Definition 5.3 of \cite{olver12mf} (see also, \cite[Sect. 8]{olver99c}), we give the following definition.
\begin{definition} 
For $M_I$ defined in (\ref{eq-MI}), we call it \emph{a regular coordinate intersection} if $M_I$ satisfies Condition A above.
Moreover, suppose that $\dim M_I=d$. If $\dim Gs\cap L_I=0$, the regular coordinate intersection $M_I$ is called  \emph{a $d$-dimensional local coordinate cross-section}. 
If a $d$-dimensional local coordinate cross-section intersects each orbit at most once, we call it \emph{a $d$-dimensional regular coordinate cross-section.}
\end{definition}

The main difference between local coordinate cross-section and regular coordinate cross-section is that the same group
orbit may intersect $M_I$ more than once \cite{olver95e,olver99c,olver12mf}.

\subsection{The relation between the tangent space and the nullspace}
\

For the polynomial system $F(x)=0$ as in (\ref{eq-FPE}), let $V(F)\subseteq \mathbb{R}^n$ be the set of solutions. 
We say that $G$ is an \emph{isotropy group} (or symmetry group) of the solution set $V(F)$ if the solution set $V(F)$ is invariant under the group action of $G$, that is,
$$g(V(F))=V(F)~\text{for~all}~g\in G.$$
Here we use the term ``isotropy group'' \cite{Bur15,deGr78-1}, which is called \emph{the symmetry group} in \cite{olver93li,olver95e,olver99c}.

If a Lie group $G$ is an isotropy group of a solution set $V(F)$, whose Lie algebra is $\mathfrak{g}$, by Theorem 2.8 of \cite{olver93li} (see also \cite[Theorem 2.79]{olver95e}), we have the following property.
To be self-contained, we give the proof here.
\begin{proposition}\label{prop-jas}
For a solution $s\in V(F)$, let $J(s)\in \mathbb{R}^{m\times n}$ be the corresponding Jacobian matrix as in (\ref{eq-jac}). 
Then we have
$$J(s)(A\cdot s)=0,~where~A\in \mathfrak{g},$$
where $A\cdot s\in \mathbb{R}^n$ is the infinitesimal generator defined in (\ref{eq-infgen}).
\end{proposition}
\begin{proof}
Since $G$ is an isotropy group of $V(F)$, we have
$$
F(gs)=0 \quad \text{whenever } F(s)=0, \; g \in G.
$$

Let $A \in \mathfrak{g}$ and consider the one-parameter subgroup $g(t)=\exp(tA)$ in $G$, where $t~\in \mathbb{R}$. Then
\begin{equation}\label{eq-fexpt}
F(\exp(tA)s)=0,
\end{equation}
since $s \in V(F)$ and $G$ preserves the solution set.
In (\ref{eq-fexpt}), taking derivative with respect to $t$ on both sides, 
and by the chain rule we have
$$
\frac{d}{dt} F(\exp(tA)s)
=
J(\exp(tA)s) \frac{d}{dt}(\exp(tA)s)=0.
$$
Then evaluating at $t=0$, we obtain
$$
J(s) \left.\frac{d}{dt}(\exp(tA)s)\right|_{t=0}
= J(s)(A\cdot s)=0.
$$
\end{proof}

Let $d=\operatorname{rank}N(s)$ and $d_{N_I}=\operatorname{rank}N_I(s)$ be the rank of $N(s)$ and $N_I(s)$ defined in (\ref{eq-ns}) and (\ref{eq-nis}), respectively.
Let $d_T=\operatorname{rank}T(s)$ and $d_{T_I}=\operatorname{rank}T_I(s)$ be the rank of $T(s)$ and $T_I(s)$ defined in (\ref{eq-Ts}) and (\ref{eq-Tis}), respectively.
By Proposition \ref{prop-jas}, we have the following corollary.
\begin{corollary}\label{cor-tti}
Let $F(x)=0$ be a polynomial system with solution set $V(F)$. Suppose that a Lie group $G$ is an isotropy group of $V(F)$. For a solution $s\in V(F)$, the tangent space $T_s (Gs)$ is a linear subspace of the nullspace $\mathrm{Null}(s)$. For any $I\subseteq [n]$, we have
\begin{equation}\label{eq-dnt}
d-d_{N_I}\geq d_T-d_{T_I},
\end{equation}
and $d_{N_I}\geq d_{T_I}$.
In particular, if $d=d_T$, then we have $d_{N_I}= d_{T_I}$ and $d-d_{N_I}= d_T-d_{T_I}$.
\end{corollary}
\begin{proof}
It follows from Lemma \ref{lem-subfull} and Corollary \ref{cor-subfull}.
\end{proof}

\subsection{Making the substitution and modulo the Lie group action}
\label{subsec-ntmethod}
\

We continue the discussions in Section \ref{subsec-first}. Suppose that
$F(x)=0$ is a polynomial system with a positive-dimensional solution set
$V(F)$, and $G$ is an isotropy group of $V(F)$. Suppose that $s\in V(F)$ is smooth and belongs to a positive-dimensional solution component. Let $s=(s_1,s_2,\ldots,s_n)$. Let $N(s)$ be the nullspace basis matrix of $s$ defined in (\ref{eq-ns}). Let $T(s)$ be the tangent basis matrix of $s$ defined in (\ref{def-tsm}). Choosing a subset $I\subseteq [n]$, we obtain the corresponding submatrices $N_I(s)$ and $T_I(s)$ of  $N(s)$ and $T(s)$, respectively. Let $s_I$ be the partial solution of $s$ that we want to fix and make a substitution. Let $L_I$ be the corresponding affine linear subspace defined in (\ref{eq-LI}). Let $F_I(x)=0$ be the corresponding reduced polynomial system defined in (\ref{eq-FI}).

In Corollary \ref{cor-tti}, (\ref{eq-dnt}) can be rewritten as 
\begin{equation}\label{eq-ngeqt}
\operatorname{rank}N(s)-\operatorname{rank}N_I(s)\geq \operatorname{rank}T(s)-\operatorname{rank}T_I(s),
\end{equation}
for any $I\subseteq [n]$. The problem is how to choose an index set $I$ such that $V(F_I)$ is positive-dimensional and contains solutions in distinct group orbits. In the following, we give a practical condition, which can help us to find the available candidate index set $I$.

After choosing $I$, by (\ref{eq-ngeqt}) two conditions will happen. The first one is:
\begin{equation}\label{eq-ntok}
\operatorname{rank}N(s)-\operatorname{rank}N_I(s)>\operatorname{rank}T(s)-\operatorname{rank}T_I(s),
\end{equation}
which is a natural first-order criterion for selecting a candidate index set
$I$. It implies that the reduced system has linear directions at $s$
which do not belong to the tangent space of the orbit slice.
If a neighbourhood of $s$ in $V(F_I)$ admits a regular foliation
whose leaves are orbit slices with constant dimension, then
(\ref{eq-ntok}) implies that this neighbourhood intersects infinitely many group orbits. A typical case is
\begin{equation*} 
\operatorname{rank}T(s)-\operatorname{rank}T_I(s)=0.
\end{equation*}
Under the transverse intersection assumption, the orbit slice $Gs\cap L_I$
is then zero-dimensional at $s$. The corresponding reduced solution
set provides a local coordinate cross-section \cite{Hubert07sa,Kogan23inv,olver12mf}. If it intersects each nearby orbit at most once, it is a regular coordinate cross-section.

The second case is: 
$$\operatorname{rank}N(s)-\operatorname{rank}N_I(s)=\operatorname{rank}T(s)-\operatorname{rank}T_I(s).$$
In this case, the first-order condition cannot tell us if there are 
linear directions at $s$ which do not belong to the tangent space of the orbit slice. So in this case,  we do not use $I$ as a practical candidate.

Although the above discussions are assumptions, after making a substitution, the reduced polynomial system $F_I(x)=0$ can be solved.
Then we can verify our assumptions. For example,  for $x\in V(F_I)$ we can check if $\operatorname{rank}T(x)-\operatorname{rank}T_I(x)=\operatorname{rank}T(s)-\operatorname{rank}T_I(s)$ and decide if $\dim_x V(F_I)>\operatorname{rank}T(x)-\operatorname{rank}T_I(x)$, generically. In particular, for Brent equations, we can use the current method to see if they are in the same orbit \cite{berger22,Heule21,kaumoo22,Tichav-21}.
When $\operatorname{rank}T(s)=\operatorname{rank}T_I(s)$, the solution set may provide a local coordinate cross-section. After solving the reduced polynomial system we can check if the Lie group action is quotiented, locally or globally. In the following sections, we will discuss the application of (\ref{eq-ntok}) to Brent equations.

\subsection{A method for finding the candidate index sets}\label{subs-findI}
\

In this section, we present a method for finding a
candidate index set $I\subseteq[n]$ satisfying inequality~(\ref{eq-ntok}).

A vector $v\in\mathbb{R}^n$ is called \emph{sparse} if most of its
entries are zero. Suppose that $I\subseteq[n]$ has large cardinality
and satisfies inequality~(\ref{eq-ntok}). Then there exists a vector
$$v\in \operatorname{Null}(s)\setminus T_s(G\cdot s)$$
such that $v_I=0$. Since $|I|$ is large, the vector $v$ is sparse.
Thus, one possible approach to finding $I$ is by solving suitable linear equations and searching for sparse solutions. In this paper, however, we use another practical greedy procedure, which is described below.

For $I\subseteq[n]$, define
\begin{equation}\label{eq-DI}
D(I)=
\bigl(\operatorname{rank}N(s)-\operatorname{rank}N_I(s)\bigr)
-\bigl(\operatorname{rank}T(s)-\operatorname{rank}T_I(s)\bigr).
\end{equation}
We call $D(I)$ the \emph{rank gap associated to $I$}. For every
$I\subseteq[n]$, we have
$$0\leq D(I)\leq
\operatorname{rank}N(s)-\operatorname{rank}T(s).$$
In particular, when $I=\emptyset$, we have 
$D(I)=\operatorname{rank}N(s)-\operatorname{rank}T(s)$; when $I=[n]$, we have $D(I)=0$. Suppose that we want to find $I$ such that $|I|$ is large enough and the gap 
$$D(I)\geq k,$$
where
$$0<k\leq \operatorname{rank}N(s)-\operatorname{rank}T(s).$$
Our (greedy) search proceeds as follows. Starting from
$I_0=\emptyset$, at each step $i$, we choose a
$u_i\in [n]\setminus I_i$ such that 
$$D(I_{i+1})\geq k,$$
where $I_{i+1}=I_i\cup\{u_i\}$.
The procedure terminates when no such  $u_i$ exists.
For Brent equations, the Julia codes implementing this method can be found in \cite{L-23}.

\section{The tangent and nullspace basis matrices for the Brent equations}
\label{sec-tnsbe}

In this section, we discuss how to compute the tangent and nullspace basis matrices for a solution of the Brent equations. We first recall the definition of the Brent equations and the associated isotropy group.

Let  $\langle m,n,p\rangle$ denote the matrix multiplication tensor
for multiplying an $m\times n$ matrix by an $n\times p$ matrix \cite{Bur15,dumas26,Lan17}. It is given by
$$
\langle m,n,p\rangle
=\sum_{i=1}^{m}\sum_{j=1}^{n}\sum_{k=1}^{p}e_{ij}\otimes e_{jk}\otimes e_{ki}\in \mathbb{R}^{m\times n}\otimes\mathbb{R}^{n\times p}\otimes\mathbb{R}^{p\times m},
$$
where $e_{ij}$ denotes the matrix with a 1 in the $i$th row and $j$th column, and other entries are zero. 
A rank-$r$ tensor decomposition of $\langle m,n,p\rangle$ can be expressed by
\begin{equation}\label{eq-mnpr}
\langle m,n,p\rangle
=\sum_{t=1}^{r} X_{t}\otimes Y_{t}\otimes Z_{t},  
\end{equation}
where
\begin{equation}\label{eq-brentxyz}
X_{t}=\left(x^{(t)}_{i_1,i_2}\right)\in \mathbb{R}^{m\times n},\qquad
Y_{t}=\left(y^{(t)}_{j_1,j_2}\right)\in \mathbb{R}^{n\times p},\qquad
Z_{t}=\left(z^{(t)}_{k_1,k_2}\right)\in \mathbb{R}^{p\times m},
\end{equation}
and $k_2,i_1\in \{1,2,\dots,m\}$, $i_2,j_1\in\{1,2,\dots,n\}$ and
$j_2,k_1\in\{1,2,\dots,p\}$.

Finding a tensor decomposition of $\langle m,n,p \rangle$ with rank $r$ as in (\ref{eq-mnpr}) is equivalent to finding solutions of a system of polynomial equations called the \emph{Brent equations} \cite{Brent70,Heule21,Smir-13}, which is given as follows:
\begin{equation}\label{eq-brent}
\sum_{t=1}^{r} x_{i_1,i_2}^{(t)}y_{j_1,j_2}^{(t)}z_{k_1,k_2}^{(t)}=
\delta_{i_2,j_1}\delta_{j_2,k_1}\delta_{k_2,i_1},
\end{equation}
where $\delta_{ij}$ is the Kronecker symbol, $x_{i_1,i_2}^{(t)}$, $y_{j_1,j_2}^{(t)}$ and $z_{k_1,k_2}^{(t)}$ are viewed as unknown variables.

We call the polynomial system in (\ref{eq-brent}) \emph{the Brent equations of type $[m,n,p|r]$}, and denote it by  $B(m,n,p|r)$. Thus, the polynomial system $B(m,n,p|r)$ has $(mn+np+pm)r$ variables and $(mnp)^2$ equations. Usually, $B(m,n,p|r)$ is an overdetermined polynomial system. The solution set of $B(m,n,p|r)$ in $\mathbb{R}^{(mn+np+pm)r}$ is denoted by $\mathrm{V}(m,n,p|r)$. Using the same symbol, $B(m,n,p|r)$ can also be viewed as a polynomial mapping:
$$B(m,n,p|r)\colon \mathbb{R}^{(mn+np+pm)r}\to \mathbb{R}^{(mnp)^2}.$$

For later use, we specify an ordering of the coordinates of
$\mathbb{R}^{(mn+np+pm)r}$ corresponding to the variables in
(\ref{eq-brentxyz}).

\begin{notation}\label{nota-coord}
For any matrix $U=(u_{i,j})\in\mathbb{R}^{a\times b}$, the well-known columnwise vectorization operation
$$\mathrm{Vec} \colon \mathbb{R}^{a\times b} \to \mathbb{R}^{ab\times1}$$
is defined by concatenating the columns of $U$ as a column vector of $\mathbb{R}^{ab\times1}$ (or $\mathbb{R}^{ab}$). So if $v=\mathrm{Vec}(U)$, we have $v_{(j-1)a+i} =u_{i,j}$.
For $t=1,2,\dots,r$, let
$$\mathbf{v}_t=\left(
        \begin{array}{c}
         \mathrm{Vec}(X_t) \\
         \mathrm{Vec}(Y_t) \\
         \mathrm{Vec}(Z_t) \\
        \end{array}
      \right)\in \mathbb{R}^{mn+np+pm}.$$
Then the variables in (\ref{eq-brentxyz}) can be ordered by
\begin{equation}\label{eq-coordbrentvar}
w=\left(
    \begin{array}{c}
      \mathbf{v}_1 \\
      \mathbf{v}_2 \\
       \vdots\\
       \mathbf{v}_r\\
    \end{array}
  \right)=\left(
    \begin{array}{c}
      w_1 \\
      w_2 \\
      \vdots\\
      w_k\\
    \end{array}
  \right)
  \in \mathbb{R}^{(mn+np+pm)r},
\end{equation}
where $k=(mn+np+pm)r$. So for example, we have $w_{(t-1)(mn+np+pm)+(j-1)m+i}=x_{i,j}^{(t)}$. 
\end{notation}

For solutions in $\mathrm{V}(m,n,p|r)$, we can express them in two typical forms. The first is \emph{the matrix form}. Let
$\{U_{t}\}_{t=1}^r\subseteq \mathbb{R}^{m\times n}$, $\{V_{t}\}_{t=1}^r\subseteq \mathbb{R}^{n\times p}$ and $\{W_{t}\}_{t=1}^r\subseteq \mathbb{R}^{p\times m}$,
such that
\begin{equation*}
\langle m,n,p\rangle
=\sum_{t=1}^{r} U_{t}\otimes V_{t}\otimes W_{t}.
\end{equation*}

Then by (\ref{eq-mnpr}) we obtain a solution in $\mathrm{V}(m,n,p|r)$ denoted by $s$, and  write it in matrix form as
\begin{equation}\label{eq-suvw}
s=(U_1,V_1,W_1,\ldots,U_r,V_r,W_r).
\end{equation}
In (\ref{eq-suvw}), we call $(U_i,V_i,W_i)$ the \emph{$i$-th layer} of $s$.

The second is \emph{the vector form}.  Using Notation \ref{nota-coord}, we express $s$ in vector form. It is given as follows. For $t=1,2,\ldots,r$, define
\begin{equation*}
\mathbf{s}_t=\left(
   \begin{array}{c}
         \mathrm{Vec}(U_t) \\
         \mathrm{Vec}(V_t) \\
         \mathrm{Vec}(W_t) \\
          \end{array}  \right)\in \mathbb{R}^{mn+np+pm}.
\end{equation*}
Then for $s$ in (\ref{eq-suvw}) we define the column-wise vectorization operation which is also denoted by "\textrm{Vec}"  as follows:
\begin{equation}\label{eq-vecuvw}
\mathrm{Vec}(s)
=\left(
    \begin{array}{c}
      \mathbf{s}_1 \\
      \mathbf{s}_2 \\
      \vdots\\
      \mathbf{s}_r\\
    \end{array}
  \right)
=\left(
    \begin{array}{c}
      s_1 \\
      s_2 \\
      \vdots\\
      s_k\\
    \end{array}
  \right)\in \mathbb{R}^{(mn+np+pm)r},
  \end{equation}
where $k=(mn+np+pm)r$. Then without of confusion, we identify 
$\mathrm{Vec}(s)$ in (\ref{eq-vecuvw}) with  $s$ in (\ref{eq-suvw}). 
The vector form above is used when we choose the fixed partial solutions.

\begin{remark}
There are other forms for solutions in $\mathrm{V}(m,n,p|r)$, for example, the LRP representation used in \cite{dumas25,dumas26}.
\end{remark}

\subsection{The isotropy group action on $\mathrm{V}(m,n,p|r)$}\label{subsec-isogroup}
\

It is well-known that there are four types of group actions which preserve $\mathrm{V}(m,n,p|r)$. In this paper, we call the group generated by these groups \emph{the isotropy group} of $\mathrm{V}(m,n,p|r)$, see for example \cite{Bur15,deGr78-1,Heule21}. Let $s\in\mathrm{V}(m,n,p|r)$ be a solution and write it in the matrix form defined in (\ref{eq-suvw}). Then the four types of groups and corresponding actions on $s$ are given as follows.

(1) \textbf{The de Groote action}

Let 
\begin{equation}\label{eq-gmnp}
\mathbf{G}(m,n,p)=\mathrm{GL}_m(\mathbb{R})\times \mathrm{GL}_n(\mathbb{R}) \times \mathrm{GL}_p(\mathbb{R})
\end{equation}
denote the direct product of $\mathrm{GL}_m(\mathbb{R})$, $\mathrm{GL}_n(\mathbb{R})$ and  $ \mathrm{GL}_p(\mathbb{R})$. For $(A,B,C) \in \mathbf{G}(m,n,p)$, it gives a group action $T(A,B,C)$ on $s$ defined by 
\begin{equation}\label{eq-degroote}
T(A,B,C)\cdot s=\left(AU_1B^{-1},BV_1C^{-1},CW_1A^{-1},\ldots,
AU_rB^{-1},BV_rC^{-1},CW_rA^{-1}
\right)
\end{equation}
We call the group action defined in (\ref{eq-degroote}) the \emph{de Groote action} (or the sandwiching action \cite{kaumoo22}). In this case, $\mathbf{G}(m,n,p)$ is called \emph{the de Groote group} of type $(m,n,p)$. 
It is well-known that (see, e.g., \cite{olver93li}) the dimension of $\mathbf{G}(m,n,p)$ is 
\begin{equation*}
\dim \mathbf{G}(m,n,p)=m^2+n^2+p^2.
\end{equation*}
It is not hard to see that
the subgroup
\begin{equation}\label{eq-g0mnp}
\mathbf{G}_0(m,n,p)=\{(\lambda I_m,\lambda I_n, \lambda I_p)|\lambda\in \mathbb{R}^{\times} \}\subseteq \mathbf{G}(m,n,p)
\end{equation}
belongs to the stabilizer of de Groote action for all $s$, whose dimension is one.

(2) \textbf{Layer scaling action}

Let 
\begin{equation*} 
\mathbf{L}(2r)=\{(\lambda_1,\mu_1,\ldots,\lambda_r,\mu_r)\in \mathbb{R}^{2r}|~\text{where}~\lambda_i,\mu_i\in \mathbb{R}^\times~\text{and}~i=1,2,...,r\}.
\end{equation*}
Then we can see that $\mathbf{L}(2r)$ is a Lie group of dimension $2r$, which is called \emph{the layer scaling group}.
For $L \in \mathbf{L}(2r)$, it acts on $s$ by 
\begin{equation*}
L \cdot s=\left( \lambda_1 U_1, \mu_1 V_1, (\lambda_1\mu_1)^{-1}W_1, \ldots,
\lambda_r U_r, \mu_r V_r, (\lambda_r\mu_r)^{-1}W_r
\right).
\end{equation*}
\begin{remark}\label{rem-overlap}
Let 
\begin{equation*} 
\mathbf{L}_0(2r)=\{(\lambda,\mu,\ldots,\lambda,\mu)\in \mathbb{R}^{2r}|~\text{where}~\lambda,\mu\in \mathbb{R}^\times\}\subseteq
\mathbf{L}(2r)
\end{equation*}
be the subgroup of $\mathbf{L}(2r)$ such that $\lambda=\lambda_i$ and 
$\mu=\mu_i$ for $i=1,2,\ldots,r$. Then the group action of $\mathbf{L}_0(2r)$ overlaps with a subgroup of $\mathbf{G}(m,n,p)$. 
In fact, it can be obtained from $\mathbf{G}(m,n,p)$ by setting $A=\lambda I_m$, $B= I_n$ and $C=\mu^{-1} I_p$. It is not hard to see that the dimension of $\mathbf{L}_0(2r)$ is 2.
\end{remark}

(3) \textbf{The action by permutation of rank-one summands}

Let $S_r$ denote the permutation group of order $r$. 
For $\sigma\in S_r$, define an action on $s$ by
\begin{equation*} 
\sigma\cdot s=  \left(U_{\sigma^{-1}(1)},V_{\sigma^{-1}(1)},W_{\sigma^{-1}(1)},\ldots,
U_{\sigma^{-1}(r)},V_{\sigma^{-1}(r)},W_{\sigma^{-1}(r)}\right).
\end{equation*}
We call this type of isotropy group \emph{the permutation summands group} and denote it by $S_r$.

(4) \textbf{The action by the $S_3$-symmetry}

This type of isotropy group is special. When $m\neq n\neq p$, this type of group action is trivial (see, e.g., Theorem 4.12 of \cite{Bur15}). Here, we just give the (typical) case: $m=n=p$.

When $m=n=p$, the matrix multiplication tensor is $\langle n,n,n\rangle $,
so three tensor factors have the same shape. There is an additional discrete symmetry. The symmetry group is isomorphic to $S_3$. It is generated by one cyclic symmetry and one transpose-swap symmetry.

First, define the cyclic symmetry action
$$
\theta\cdot s=(W_1,U_1,V_1,\ldots, W_r,U_r,V_r).
$$
Then we have that $\theta$ has order $3$.
Second, define the transpose-swap symmetry
$$\tau\cdot s=(U_1^{t},W_1^{t},V_1^{t},\ldots, U_r^{t},W_r^{t},V_r^{t}).
$$
Then $\tau$ has order $2$. Together, $\theta$ and $\tau$ generate the group that is isomorphic to  $S_3$. We call this type of isotropy group \emph{the $S_3$-symmetry group}. Since it depends on $m$, $n$ and $p$, we denote it by $S(m,n,p)$.

The (full) isotropy group of $\mathrm{V}(m,n,p|r)$ is generated by the four types of groups introduced above. The groups $\mathbf{G}(m,n,p)$ and
$\mathbf{L}(2r)$ are Lie groups. $S_r$ and $S(m,n,p)$ are discrete groups.

Let 
\begin{equation}\label{eq-Imnpr1}
\mathcal{I}_1(m,n,p|r)=\mathbf{G}(m,n,p)\times \mathbf{L}(2r)
\end{equation}
be the direct product of the de Groote and layer scaling groups. If we let $\mathcal{I}(m,n,p|r)$ denote the isotropy group, then from discussions above we have \cite{Bur15,deGr78-1}:
\begin{equation}\label{eq-Inmpr}
\mathcal{I}(m,n,p|r)=\mathcal{I}_1(m,n,p|r)\rtimes (S_r \times S(m,n,p)).
\end{equation}
Let 
$$\operatorname{Stab}_{\mathcal{I}}(s)$$
denote the stabilizer of $s$ under the action of
$\mathcal{I}(m,n,p|r)$. Then $\mathbf{G}_0(m,n,p)$ is a (one-dimensional) subgroup of $\operatorname{Stab}_{\mathcal{I}}(s)$.

Given a solution $s$, let $\mathcal{I}(m,n,p|r)\cdot s$ be the group orbit under the isotropy group action. Then we have the following proposition.
\begin{proposition}\label{prop-dimIs}
Let $\dim \mathcal{I}(m,n,p|r)\cdot s$ be the dimension of $\mathcal{I}(m,n,p|r)\cdot s$, we have
\begin{equation*}
\dim \mathcal{I}(m,n,p|r)\cdot s \leq
m^2+n^2+p^2+2r-3.
\end{equation*}
If $$\dim \mathcal{I}(m,n,p|r)\cdot s
=m^2+n^2+p^2+2r-3,$$ then $s$ has no other positive-dimensional stabilizer except $\mathbf{G}_0(m,n,p)$.
\end{proposition}
\begin{proof}
By definition (\ref{eq-Inmpr}), we have
$$\dim \mathcal{I}(m,n,p|r)=m^2+n^2+p^2+2r.$$
From Remark \ref{rem-overlap}, there is an overlapped group $\mathbf{L}_0(2r)$ in $\mathbf{G}(m,n,p)$ and $\mathbf{L}(2r)$, which is two-dimensional. Then in the isotropy group orbit $\mathcal{I}(m,n,p|r)\cdot s$, the overlapped part should be considered as one.

After subtracting the overlapped two-dimensional part, by Corollary \ref{cor-dimos} we have 
\begin{align*}
  \dim \mathcal{I}(m,n,p|r)\cdot s &=\dim \mathcal{I}(m,n,p|r)-\dim \operatorname{Stab}_{\mathcal{I}}(s)-2.
\end{align*}

Since $\mathbf{G}_0(m,n,p)\subseteq \operatorname{Stab}_{\mathcal{I}}(s)$ is one-dimensional, we have 
$$\dim \operatorname{Stab}_{\mathcal{I}}(s)\geq1.$$
Thus, we have the upper bound
\begin{align*}
  \dim \mathcal{I}(m,n,p|r)\cdot s \leq m^2+n^2+p^2+2r-3.
\end{align*}

So if  the equality holds, then $s$ has no other positive-dimensional stabilizer.
\end{proof}

\subsubsection{Geometry of the isotropy group action in affine coordinate space}\label{subsec-geometry}
\

Given a solution $s\in \mathrm{V}(m,n,p|r)$ with corresponding vector form 
as in (\ref{eq-vecuvw}),  let $V\subseteq \mathrm{V}(m,n,p|r)$ be the irreducible variety that contains $s$.
For the four types of isotropy groups of $\mathrm{V}(m,n,p|r)$, if we look at their actions in $\mathbb{R}^{(mn+np+pm)r}$, then some remarks should be given.

First, let $\mathcal{I}_1(m,n,p|r)=\mathbf{G}(m,n,p)\times \mathbf{L}(2r)$
be the positive-dimensional Lie group defined in (\ref{eq-Imnpr1}). So  under the $\mathcal{I}_1(m,n,p|r)$-action, we obtain a (smooth) positive-dimensional group orbit $\mathcal{I}_1(m,n,p|r)\cdot s\subseteq V$. 
When $\dim V> \dim \mathcal{I}_1(m,n,p|r)\cdot s$, the smooth part of $V$ becomes a foliation whose leaves are the group orbits $\mathcal{I}_1(m,n,p|r)\cdot s$.

On the other hand, let
\begin{equation}\label{eq-I2}
\mathcal{I}_2(m,n,p|r)=S_r\times S(m,n,p),
\end{equation}
be the permutation summands and $S_3$-symmetry groups. Then $\mathcal{I}_2(m,n,p|r)$ is a discrete group. The group orbit $\mathcal{I}_2(m,n,p|r)\cdot s$ may not belong to $V$. The reason is that, under the action of $\mathcal{I}_2(m,n,p|r)$, we obtain a  symmetry copy of $V$ in $\mathbb{R}^{(mn+np+pm)r}$. So in this case, $\mathcal{I}_2(m,n,p|r)\cdot s$ may lie in another symmetry copy of $V$.

\subsection{The computation of tangent and nullspace basis matrices}
\

Let $s\in \mathrm{V}(m,n,p|r)$ be a solution. Recall the Lie group part of the isotropy group $$\mathcal{I}_1(m,n,p|r)=\mathbf{G}(m,n,p)\times\mathbf{L}(2r).$$ 
Let $T(s)$ be the tangent basis matrix of $s$ with respect to $\mathcal{I}_1(m,n,p|r)$, which is defined in (\ref{eq-Ts}). Let $N(s)$ be the nullspace basis matrix of $s$ defined in (\ref{eq-ns}).
In this section, we discuss how to compute them. 

\subsubsection{The computation of the tangent basis matrix}
\

First, we discuss the computation of $T(s)$. Since
$$\mathcal{I}_1(m,n,p|r)=\mathbf{G}(m,n,p)\times\mathbf{L}(2r)=\mathrm{GL}_m(\mathbb{R})\times \mathrm{GL}_n(\mathbb{R}) \times \mathrm{GL}_p(\mathbb{R})\times \mathbf{L}(2r),$$
the corresponding Lie algebra is
$$\mathfrak{g}_{\mathcal{I}_1}=\mathfrak{gl}_m(\mathbb{R})\oplus \mathfrak{gl}_n(\mathbb{R}) \oplus \mathfrak{gl}_p(\mathbb{R})\oplus \mathfrak{l}(2r).$$

It is well known that
$\mathfrak{gl}_m (\mathbb{R})=\mathbb{R}^{m\times m}$ and $\mathfrak{l}(2r)=\mathbb{R}^{2r}$.
For $i,j=1,2,\ldots m$, let $E_{i,j}^{(m)}\in \mathbb{R}^{m\times m}$ denote the matrix with 1 in the $i$-th row and $j$-th column, and other entries zero.  Then $\{E_{i,j}^{(m)}|i,j=1,...,m\}$ is a standard basis of $\mathfrak{gl}_m(\mathbb{R})$. Similarly, we define the standard bases 
$\{E_{i,j}^{(n)} \}$ and $\{E_{i,j}^{(p)} \}$ of  
$\mathfrak{gl}_n(\mathbb{R})$ and $\mathfrak{gl}_p(\mathbb{R})$, respectively.
For $1\leq i\leq 2r$, let $E_{i}\in \mathbb{R}^{2r}$ denote the vector with 1 in the $i$th row and zeros elsewhere. Then $\{E_i|i=1,2,\ldots,2r\}$ is a standard basis of the abelian Lie algebra $\mathfrak{l}(2r)=\mathbb{R}^{2r}$.

Under the canonical embeddings of the summands into the direct sum, we have that the linear span of 
$$\{E_{i,j}^{(m)}\}_{i,j=1}^m\cup\{E_{i,j}^{(n)}\}_{i,j=1}^n\cup
\{E_{i,j}^{(p)}\}_{i,j=1}^p\cup\{E_i\}_{i=1}^{2r}$$
is equal to $\mathfrak{g}_{\mathcal{I}_1}$.

To compute $T(s)$, we first derive the infinitesimal generators in
the matrix form of $s$ defined in \eqref{eq-suvw}. We then express
them in vector form and use the resulting vectors as the columns of
$T(s)$.

For $(A,B,C)\in \mathbf{G}(m,n,p)$, by (\ref{eq-degroote}) we have
\begin{equation}\label{eq-tabc}
T(A,B,C)=T(A,I_n,I_p)T(I_m,B,I_p) T(I_m,I_n,C).
\end{equation}
Let $(U,V,W)$ be an arbitrary layer of $s$, where $U\in \mathbb{R}^{m\times n}$, $V\in \mathbb{R}^{n\times p}$ and $W\in \mathbb{R}^{p\times m}$.
For $\lambda,\mu\in \mathbb{R}^{\times}$ define \emph{the $\lambda$-action} by 
\begin{equation}\label{eq-lamb}
(\lambda,1)(U, V, W)=(\lambda U, V, \lambda^{-1}W)
\end{equation}
and \emph{the $\mu$-action} by
\begin{equation}\label{eq-mu}
(1,\mu)(U, V, W)=( U, \mu V, \mu^{-1}W).
\end{equation}
Then the layer scaling action on a single layer $(U, V, W)$ is given by
$$(\lambda,\mu)(U, V, W)=(\lambda,1)(1,\mu)(U, V, W)=(\lambda U, \mu V, (\lambda\mu)^{-1}W).$$
Thus, from (\ref{eq-tabc})-(\ref{eq-mu}) there are five types of one-parameter subgroup actions on $(U,V,W)$. 

In the following, we first compute the infinitesimal generators on
a single layer $(U,V,W)$ and then apply the resulting formulas
componentwise to all $r$ layers of $s$.

\textbf{(1):} For $A \in \mathfrak{gl}_m(\mathbb{R})$ and $\mathcal{A}(t)=\exp(tA)$ where $t\in \mathbb{R}$, we have 
\begin{equation*}
A\cdot U =\left.\frac{d}{dt}\exp(tA)U\right|_{t=0}=AU,
\end{equation*}
that is, $A\cdot U$ is the multiplication of $A$ and $U$.
Similarly, we have
$$\left.\frac{d}{dt}W\exp(-tA)\right|_{t=0}=-WA.$$
Then, for the one-parameter subgroup action defined by
\begin{equation*}
\mathcal{A}(t) (U,V,W) = (\mathcal{A}(t)U, V, W\mathcal{A}(-t))=\left(\exp(t A)U,V,W\exp(-t A)\right),
\end{equation*}
we have the infinitesimal generator
\begin{equation}\label{eq-infslocauvw}
A\cdot(U,V,W)=\left.\frac{d}{dt}\mathcal{A}(t)(U,V,W)\right|_{t=0}=(AU, 0,
-WA).
\end{equation}

Thus, for $\mathcal{A}(t)=\exp(tA)$ and the one-parameter subgroup action
$$\mathcal{A}(t)s=
(\exp(tA)U_1,V_1,W_1\exp(-tA),\ldots,\exp(tA)U_r,V_r,W_r\exp(-tA)),
$$
from (\ref{eq-infslocauvw}) the corresponding infinitesimal generator is given by
\begin{equation}\label{eq-AS}
A\cdot s=(AU_1, 0 ,-W_1A,\ldots, AU_r, 0, -W_r A).
\end{equation}

\textbf{(2):} For $B \in \mathfrak{gl}_n(\mathbb{R})$, let $\mathcal{B}(t)=\exp(tB)$ where $t\in \mathbb{R}$.
Then for the one-parameter subgroup action
$$\mathcal{B}(t)s=
(U_1\exp(-tB),\exp(tB)V_1,W_1,\ldots,U_r\exp(-tB),\exp(tB)V_r,W_r),
$$
the corresponding infinitesimal generator is given by
\begin{equation}\label{eq-BS}
B\cdot s=
(-U_1 B, B V_1, 0,\ldots, -U_r B, B V_r, 0).
\end{equation}

\textbf{(3):} For  $C \in \mathfrak{gl}_p(\mathbb{R})$, let $\mathscr{C}(t)=\exp(tC)$ where $t\in \mathbb{R}$.
Then considering the one-parameter subgroup action
$$\mathscr{C}(t)s=
(U_1,V_1\exp(-tC),\exp(tC)W_1,\ldots,U_r,V_r\exp(-tC),\exp(tC)W_r),
$$
the corresponding infinitesimal generator is given by
\begin{equation}\label{eq-CS}
C\cdot s=
(0, -V_1 C, C W_1,\ldots, 0, -V_r C, C W_r).
\end{equation}

\textbf{(4):} Recall that $\{E_i|i=1,2,\ldots,2r\}$ is the standard basis of the Lie algebra $\mathfrak{l}(2r)=\mathbb{R}^{2r}$. 
For the $\lambda$-action defined in (\ref{eq-lamb}), define the one-parameter subgroup action
\begin{equation*}
E_1(t) (U,V,W)
=\left(\exp(t)U, V , \exp(-t)W \right).
\end{equation*}
Then we have the infinitesimal generator
\begin{equation}\label{eq-lay1loc}
\left.\frac{d}{dt} E_1(t)(U,V,W)\right|_{t=0}=( U, 0, - W)
\end{equation}
For $i=2k-1$ where $k=1,2,\ldots,r$,
define the one-parameter subgroup action as
$$E_1^{(i)}(t) s=(U_1,\cdots, V_{k-1}, W_{k-1}, \exp(t) U_k, V_k, 
\exp(-t)W_k,  U_{k+1}, \cdots,W_r)$$
Then for $i=2k-1$ from (\ref{eq-lay1loc}), the infinitesimal generator $E_i\cdot s$ is given by
\begin{align}
E_i \cdot s&=\left.\frac{d}{dt} E_1^{(i)}(t) s \right|_{t=0}\label{eq-ei2k-1}\\
&=\left.\frac{d}{dt}(U_1,\cdots, V_{k-1}, W_{k-1}, \exp(t) U_k, V_k, \exp(-t)W_k,  U_{k+1}, \cdots, W_r)\right|_{t=0}\notag\\
&=(0,\cdots, 0,  U_k, 0, -W_k, 0,\cdots,0).\notag
\end{align}

\textbf{(5):} For the $\mu$-action defined in (\ref{eq-mu}), define the one-parameter subgroup action as
\begin{equation*}
E_2(t) (U,V,W)
=\left(U, \exp(t)V , \exp(- t)W \right).
\end{equation*}
Then we have
\begin{equation}\label{eq-lay2loc}
\left.\frac{d}{dt} E_2(t)(U,V,W)\right|_{t=0}=(0,  V, -W).
\end{equation}
For $i=2k$ where $k=1,2,\ldots,r$,
define the one-parameter subgroup action as
$$E_2^{(i)}(t) s=(U_1,\cdots, V_{k-1}, W_{k-1}, U_k, \exp(t)V_k, 
\exp(-t)W_k,  U_{k+1}, V_{k+1}, \cdots, W_r)$$
Then for $i=2k$ from (\ref{eq-lay2loc}) we have
\begin{align}
E_i \cdot s&=\left.\frac{d}{dt} E_2^{(i)}(t) s \right|_{t=0}\label{eq-ei2k}\\
&=\left.\frac{d}{dt}(U_{1},\cdots, V_{k-1}, W_{k-1}, U_k, \exp(t)V_k, \exp(-t)W_k,  U_{k+1}, V_{k+1}, \cdots, W_{r})\right|_{t=0}\notag\\
&=(0,\cdots, 0,  0 , V_k, -W_k,  0,\cdots,0).\notag
\end{align}

Now in (\ref{eq-AS}), we let $A=E^{(m)}_{i,j}$ and denote
\begin{equation}\label{eq-aijs}
\mathbf{a}_{i,j}=\mathrm{Vec}(E^{(m)}_{i,j}\cdot s),
\end{equation}
where $i,j\in [m]$ and ``\textrm{Vec}'' is the vectorization operation defined in (\ref{eq-vecuvw}). Then we can see that $\mathbf{a}_{i,j}$ is a vector in 
$\mathbb{R}^{(mn+np+pm)r}$.
Similarly, in (\ref{eq-BS}) and (\ref{eq-CS}), we let $B=E^{(n)}_{i,j}$ and 
$C=E^{(p)}_{i,j}$, respectively. Then
denote
\begin{equation}\label{eq-bijs}
\mathbf{b}_{i,j}=\mathrm{Vec}(E^{(n)}_{i,j}\cdot s),\quad \text{where}~ i,j\in [n]
\end{equation}
and
\begin{equation}\label{eq-cijs}
\mathbf{c}_{i,j}=\mathrm{Vec}(E^{(p)}_{i,j}\cdot s), \quad \text{where}~  i,j\in [p].
\end{equation}

When $i=2k-1$ and $1\leq k\leq r$, for $E_{i}\cdot s$ defined in (\ref{eq-ei2k-1})  we denote
\begin{equation}\label{eq-ell1i}
\ell_1^{(k)}=\mathrm{Vec}(E_{i}\cdot s).
\end{equation}
When $i=2k$ and $1\leq k\leq r$, for
$E_{i}\cdot s$  defined in  (\ref{eq-ei2k}) we denote
\begin{equation}\label{eq-ell2i}
\ell_2^{(k)}=\mathrm{Vec}(E_{i}\cdot s).
\end{equation}

Then from (\ref{eq-aijs})-(\ref{eq-ell2i}) and Definition \ref{def-tsm},
we have the following theorem which tells us how to compute $T(s)$.
\begin{theorem}\label{thm-tbm}
The tangent basis matrix of $s\in \mathrm{V}(m,n,p|r)$ is given by
\begin{equation}\label{eq-tsmnpr}
T(s)=\left(\mathbf{a}_{1,1},\mathbf{a}_{1,2},\cdots \mathbf{a}_{m,m}, 
\mathbf{b}_{1,1},\mathbf{b}_{1,2},\cdots \mathbf{b}_{n,n},
\mathbf{c}_{1,1},\cdots \mathbf{c}_{p,p},
\ell_1^{(1)},\ell_2^{(1)}, \cdots, \ell_1^{(r)},\ell_2^{(r)} 
\right),
\end{equation}
which belongs to $\mathbb{R}^{(mn+np+pm)r\times (m^2+n^2+p^2+2r)}$.
\end{theorem}

By Corollary \ref{cor-dimos} and Proposition \ref{prop-dimIs}, we have the following corollary. It tells us that after computing the rank of $T(s)$ we can see if $s$ has other positive-dimensional stabilizer or not, which partially answers a question in the concluding remarks of \cite{li25loc}.
\begin{corollary}\label{prop-ranktsmnp}
For the tangent basis matrix $T(s)$ in (\ref{eq-tsmnpr}), we have
$$\operatorname{rank}T(s)\leq m^2+n^2+p^2+2r-3.$$
If $s$ has no other positive-dimensional stabilizer except $\mathbf{G}_0(m,n,p)$ in (\ref{eq-g0mnp}), then we have
$$\operatorname{rank}T(s)=m^2+n^2+p^2+2r-3.$$
\end{corollary}

\subsubsection{The computation of the nullspace basis matrix}
\

The computation of $N(s)$ depends on the Jacobian matrix $J(s)$. In fact,
from (\ref{eq-ns}), we know that
$$N(s)=(v_1,v_2,\ldots,v_d),$$
where $\{v_i\}_{i=1}^d$ is a basis of the nullspace.
Once $J(s)$ is known, $N(s)$ can be obtained by solving the linear equations $J(s)x=0$ and choosing a (typical) basis. In Julia, once $J(s)$ is given, we can obtain $N(s)$ from the package "Nemo.jl",  conveniently.
So we just discuss the computation of $J(s)$.

Suppose that 
$$s=(U_1,V_1,W_1,\ldots,U_r,V_r,W_r)\in \mathrm{V}(m,n,p|r)$$
is a solution. For $i=1,2,\ldots,r$, let $\mathbf{u}_i=\mathrm{Vec}(U_i)$, $\mathbf{v}_i=\mathrm{Vec}(V_i)$ and $\mathbf{w}_i=\mathrm{Vec}(W_i)$.
Just as the discussion in \cite[Sect. 2.5]{CV25}, if we take partial derivatives with respect to the coordinate order as in
(\ref{eq-vecuvw}), then the corresponding Jacobian matrix is given by
\begin{align}
J(s)=&(I_{mn}\otimes \mathbf{v}_1\otimes \mathbf{w}_1, \mathbf{u}_1\otimes I_{np}\otimes \mathbf{w}_1, 
\mathbf{u}_1\otimes \mathbf{v}_1\otimes I_{pm}, I_{mn}\otimes \mathbf{v}_2\otimes \mathbf{w}_2,\label{eq-js}\\
   & \mathbf{u}_2\otimes I_{np}\otimes \mathbf{w}_2, \mathbf{u}_2\otimes \mathbf{v}_2\otimes I_{pm},\cdots\cdots\cdots\cdots,I_{mn}\otimes \mathbf{v}_r\otimes \mathbf{w}_r,\notag\\
   & \mathbf{u}_r\otimes I_{np}\otimes \mathbf{w}_r, 
\mathbf{u}_r\otimes \mathbf{v}_r\otimes I_{pm}),\notag
\end{align}
where $I_{mn}\in \mathbb{R}^{mn\times mn}$ is the identity matrix, and similarly for
$I_{np}$ and $I_{pm}$. Thus, by definition above, we have 
$$J(s)\in \mathbb{R}^{(mnp)^2\times(mn+np+pm)r}.$$

\begin{remark}
For $J(s)$ defined above and $T(s)$ in (\ref{eq-tsmnpr}), by Proposition
\ref{prop-jas}, we have 
$$J(s)T(s)=0,$$
which is the zero matrix in $\mathbb{R}^{(mnp)^2\times(m^2+n^2+p^2+2r)}$.
\end{remark}

\section{The numerical experiments and results}\label{sec-numtest}

In this section, we implement the method proposed in Section \ref{subsec-ntmethod} for Brent equations. Numerical experiment results are given. Starting from a solution, our method can efficiently produce parameterized solution sets. We first summarize the procedure used in this paper.

\subsection{A practical substitution procedure for the Brent Equations}\label{subs-subsprocedure}
\

The procedure of our numerical tests is summarized in steps as follows. It is designed for smooth solutions. Numerical tests on singular solutions are given in Section \ref{subs-testsing}.

\begin{description}
\item[Step 1] Perform a deflation test for the given solution $s\in \mathrm{V}(m,n,p|r)$ to determine if it is smooth. 
    For example, we can use the algorithm in \cite{li25def}. The deflation sequences give an estimation of the local dimensions, which implies whether the solution is smooth. If the first 4 numbers of the deflation sequence are all equal, then we assume that $s$ is smooth.

\item[Step 2] After the deflation test, if $s$ is smooth, then 
compute the Jacobian matrix $J(s)$ defined in (\ref{eq-js}), and compute the corresponding nullspace basis matrix $N(s)$. $N(s)$ can be computed in Julia by the package "Nemo.jl". Then compute the tangent basis matrix $T(s)$ defined in Theorem \ref{thm-tbm}.

\item[Step 3] Compute $\operatorname{rank}N(s)=d$ and $\operatorname{rank}T(s)=d_T$, and determine whether
$d>d_T$. If $d>d_T$, then in the neighbourhood of $s$, the solution component can be a foliation whose leaves are the group orbits.
In this case, we can find new solutions that belong to different group orbits by fixing partial solution and making the substitution.

If $d=d_T$, the solution component containing $s$ lies in an 
$\mathcal{I}(m,n,p|r)$-orbit. Hence, the solutions obtained by fixing partial solution and making the substitution may still belong to the same $\mathcal{I}(m,n,p|r)$-orbit.

\item[Step 4]
Establish a one-to-one correspondence among the column indices of $J(s)$, the row indices of $N(s)$, the row indices of $T(s)$ and the variables. In this paper, we use the correspondence between (\ref{eq-coordbrentvar}) and (\ref{eq-vecuvw}).

\item[Step 5] 
Suppose that $d>d_T$ and $k=(mn+np+pm)r$. 
Choose a subset $I\subseteq[k]$, and form the submatrices
$N_I(s)$ and $T_I(s)$. Then compute $\operatorname{rank}N_I(s)=d_{N_I}$ and $\operatorname{rank}T_I(s)=d_{T_I}$, and search for
an index set $I$ satisfying $d-d_{N_I}>d_T-d_{T_I}$ as in (\ref{eq-ntok}). Depending on the computational complexity, our searching should keep $|I|$ large enough. This step we use the method proposed in Section \ref{subs-findI}.

\item[Step 6]
After finding the index set $I$ such that $d-d_{N_I}>d_T-d_{T_I}$, we obtain a partial solution $s_I$ for fixing and substitution. Then from the original Brent equations (\ref{eq-brent}), we get the reduced Brent equations by the constraint $x_i-s_i=0$ for $i\in I$ as in (\ref{eq-FI}). Then substituting $x_i=s_i$, it becomes a small scale polynomial system.

\item[Step 7]
After substituting, we can try to solve the reduced Brent equations by the symbolic computation software that can compute the Gröbner bases, such as Oscar.jl. If we want to find the solution efficiently, the computation of the Gröbner basis is suggested to be done under the lexicographic order.

\item[Step 8]
If the solution set obtained by computing Gröbner basis is positive-dimensional, then we often get a parameterized solution set. After that, we can check if the parameterized solution set contains new solutions in different isotropy group orbits, for example, by methods in \cite{berger22,Heule21,kaumoo22,Tichav-21}.
\end{description}

\begin{remark}
In \textbf{Step} 3 above, there are solutions of $\mathrm{V}(m,n,p|r)$ such that $d=d_T$. The first one is Strassen's solution in $\mathrm{V}(2,2,2|7)$ \cite{deGr78-2,Stra-69}, which is well-known. Another one is Smirnov's solution in $\mathrm{V}(3,3,6|40)$. In fact,
the deflation sequence found in Section 5.4 of \cite{li25def} provides numerical evidence that the local dimension of Smirnov's solution is 131, which is equal to $m^2+n^2+p^2+2r-3$, where $m=n=3$, $p=6$ and $r=40$. After computing the corresponding $T(s)$, we also found that $\operatorname{rank}T(s)=131$. 

From Corollary \ref{cor-subfull}, we have that computation of the ranks  in \textbf{Step} 3 and 5 is independent of the choice of the bases. 
\end{remark}

In \textbf{Step} 8, if the parameterized solution set is obtained, then we want to know if there are new solutions in it. One method is by checking the possible nontrivial $\mathcal{I}(m,n,p|r)$-actions, such as the discussions in \cite[Sect.3]{JohnMc86}. However, as pointed out in \cite{JohnMc86}, the process is tedious. When we just want to know if there are new $\mathcal{I}(m,n,p|r)$-orbits in the parameterized solution set, there are some relatively simple conditions which can help us. In the following, we give some discussions. 

Given a solution $s\in \mathrm{V}(m,n,p|r)$, let $s_I$ be the partial solution obtained in \textbf{Step} 6 for some index set $I$.
By the substitution procedure, suppose that we obtain a one-dimensional solution set that is denoted by  $s(t)$, where $t\in \mathbb{R}$. Suppose that $s(t)$ is smooth except at finitely many singular points. Let $J\subseteq \mathbb{R}$ be an open interval on which $s(t)$ is smooth. 
Let $T(s(t))$ denote the tangent basis matrix associated with $s(t)$, that is defined in (\ref{eq-tsmnpr}). Since the partial solution is fixed, we have $s_I(t)=s_I$.
Suppose that 
\begin{equation}\label{eq-tstmnp}
\operatorname{rank}T(s(t))=\operatorname{rank}T(s)=m^2+n^2+p^2+2r-3
\end{equation}
for all $t\in J$.
Let $s'(t)$ be the derivative of $s(t)$. Adding one column on the right of $T(s(t))$ we obtain
$$\left(T(s(t)),s'(t)\right).$$
Suppose that
\begin{equation}\label{eq-tsst1}
\operatorname{rank}\left(T(s(t)),s'(t)\right)=\operatorname{rank}T(s(t))+1=m^2+n^2+p^2+2r-2,
\end{equation}
for all $t\in J$.    
Then $s(J)$ meets infinitely many distinct $\mathcal{I}(m,n,p|r)$-orbits. 

In fact, let $\mathcal{I}_1(m,n,p|r)$ denote the de Groote and layer scaling groups defined in (\ref{eq-Imnpr1}). Let
\begin{equation*}
\mathcal{I}_1(m,n,p|r)\cdot s(J)\subseteq V(m,n,p|r)
\end{equation*}
denote \emph{a saturated subset} \cite{lee12smooth}, i.e., the union of  $\mathcal{I}_1(m,n,p|r)$-orbits through points of $s(J)$. Let $q=m^2+n^2+p^2+2r-3$. By (\ref{eq-tstmnp}) each $\mathcal{I}_1(m,n,p|r)$-orbit is of dimension $q$. From (\ref{eq-tsst1}), the dimension of the immersed submanifold $\mathcal{I}_1(m,n,p|r)\cdot s(J)$ is $q+1$. Thus $s(J)$ meets infinitely many distinct $\mathcal{I}_1(m,n,p|r)$-orbits. Let $\mathcal{I}_2(m,n,p|r)$ denote the permutation summands and $S_3$-symmetry groups defined in (\ref{eq-I2}), which is a finite discrete group. The full isotropy group $\mathcal{I}(m,n,p|r)$ is a finite extension of $\mathcal{I}_1(m,n,p|r)$, since the additional actions by $\mathcal{I}_2(m,n,p|r)$ are finite. Hence each $\mathcal{I}(m,n,p|r)$-orbit is a finite union of $\mathcal{I}_1(m,n,p|r)$-orbits. Since $s(J)$ meets infinitely many distinct $\mathcal{I}_1(m,n,p|r)$-orbits, it also meets infinitely many distinct $\mathcal{I}(m,n,p|r)$-orbits. 

\begin{remark}\label{remk-invf}
There is another way to determine if there are new $\mathcal{I}(m,n,p|r)$-orbits, which is realized by constructing \emph{an invariant function} \cite{olver93li,olver99c} on the parameterized solution set. Given a one-dimensional solution $s(t)\subseteq \mathrm{V}(m,n,p|r)$,
if a nonconstant invariant function $\varphi(s(t))$ is constructed, then 
$s(t_1)$ and $s(t_2)$ belong to distinct $\mathcal{I}(m,n,p|r)$-orbits when
$\varphi(s(t_1))\neq\varphi(s(t_2))$ where $t_1$, $t_2\in \mathbb{R}$ (see e.g., \cite[Example 4.20]{olver99c}). The construction of the invariant functions for parameterized solution sets is flexible. The general theory can be found in \cite{Hubert07sa,Kogan23inv}. We will construct an invariant function in Section \ref{subs-testdps}.
\end{remark}

\subsection{The numerical results}

\subsubsection{The test on Dumas, Pernet and Sedoglavic's solution in $\mathrm{V}(4,4,4|48)$} \label{subs-testdps}
\ 

In this section, the numerical tests are based on the solution obtained by Dumas, Pernet and Sedoglavic, which is given in Appendix B of \cite{dumas25}. It is transformed from AlphaEvolve's solution \cite{novikov25ae} by a (complex) $\mathcal{I}(4,4,4|48)$-action.  

With out of confusion, let $s$ denote this solution. Recently, another rational solution is also found \cite{moran26}. They belong to the same $\mathcal{I}(4,4,4|48)$-orbit.
Other solutions in $\mathrm{V}(4,4,4|48)$ are reported in \cite{kaporin24-cmmp,kaporin24-Dokl}.

When $m=n=p=4$ and $r=48$, the Brent equations $B(4,4,4|48)$ has 
$2304$ variables and $4096$ equations.
After performing a deflation test as in \cite{li25def}, we find that the first 4 numbers of deflation sequence of the solution are
$$(151,151,151,151).$$
The constant deflation sequence provides numerical evidence that
the solution is smooth with local dimension 151. After computing the corresponding Jacobian matrix $J(s)$ defined in (\ref{eq-js}), using the Julia package "Nemo.jl" we obtain the nullspace basis matrix $N(s)$, whose rank is $151$. Computing the tangent basis matrix $T(s)$ defined in (\ref{eq-tsmnpr}), we find that
$\operatorname{rank}T(s)=141$. Then we can search an index set $I$ with $|I|=2082$, which gives a rank gap
\begin{equation}\label{eq-gap1}
1=\operatorname{rank}N(s)-\operatorname{rank}N_I(s)>\operatorname{rank}T(s)-\operatorname{rank}T_I(s)=0.
\end{equation}

Let $s_I$ be the corresponding partial solution. Then after substituting $s_I$, we obtain a reduced polynomial system, the number of its variables becomes 2304-2082=222. In the reduced polynomial system, some equations differ only by nonzero scalar multiples. So we can just choose one such equations. After doing this, the final reduced polynomial system becomes 
smaller, many of its equations are quadratic. Then
under the lexicographic order, its Gröbner basis can be easily computed by Oscar.jl \cite{Dec25} in a few seconds. Then we find that the solution set is a one-dimensional curve with rational coefficients, which is given in Appendix \ref{sec-app}. It can be easily checked that the parameterized solution $s(t)$ satisfies (\ref{eq-tstmnp}) and (\ref{eq-tsst1}). So 
$s(t)$ contains infinitely many pairwise inequivalent points. For example, we can choose $t=\frac{1}{2}$, $\frac{1}{4}$ and $\frac{1}{16}$. 

Following Remark \ref{remk-invf}, we construct an invariant function 
for the parameterized solution $s(t)$ in Appendix \ref{sec-app}.

In Appendix \ref{sec-app}, for $i=1,2,\ldots,48$ the matrix triple
$$(U_i(t),V_i(t),W_i(t))$$
is called the \emph{$i$-th layer} of $s(t)$.
For the 48 layers, consider the multiset of their rank triples
\begin{equation}\label{eq-rtriple}
\{(\operatorname{rank} U_i(t), \operatorname{rank} V_i(t),
              \operatorname{rank} W_i(t))\}_{i=1}^{48}.
\end{equation}
It is not hard to see that there are 32 layers that have rank triple 
\begin{equation}\label{eq-triple1}
(1,1,1).
\end{equation}
There are 16 layers that have rank triple 
\begin{equation}\label{eq-triple2}
(2,2,2),
\end{equation}
whose layer numbers are
$$R_2=\{4, 9, 12, 15, 17, 19, 22, 25,
      28, 32, 34, 35, 38, 40, 44, 48 \}.$$
Let
$$R_2^4=\{H~|~H\subseteq R_2~\text{and}~|H|=4\}$$
be the set of all subsets of $R_2$ that contain 4 different elements.
Then we have 
$$|R_2^4|=\binom{16}{4}=1820.$$

Let $x_1$, $x_2$, $x_3$, $x_4$ be four variables.
Suppose that $H=\{i_1,i_2,i_3,i_4\}\in R_2^4$ and $i_1<i_2<i_3<i_4$. Then consider the polynomial
\begin{equation*}
\det\left(x_1 U_{i_1}(t)+x_2 U_{i_2}(t)+x_3 U_{i_3}(t)+x_4 U_{i_4}(t)\right),
\end{equation*}
which is the determinant of $x_1 U_{i_1}(t)+x_2 U_{i_2}(t)+x_3 U_{i_3}(t)+x_4 U_{i_4}(t)$. It is a homogeneous polynomial of degree 4 with respect to $x_1$, $x_2$, $x_3$ and $x_4$.
The coefficient of the monomial $x_1x_2x_3x_4$ in it is denoted by
$$C(U,H,t)=[x_1x_2x_3x_4]
\det(x_1 U_{i_1}(t)+x_2 U_{i_2}(t)+x_3 U_{i_3}(t)+x_4 U_{i_4}(t)).$$
Then $C(U,H,t)$ is a rational function in $t$.
Similarly, the coefficients of the monomial $x_1x_2x_3x_4$ in
$\det(x_1 V_{i_1}(t)+x_2 V_{i_2}(t)+x_3 V_{i_3}(t)+x_4 V_{i_4}(t))$ 
and $\det(x_1 W_{i_1}(t)+x_2 W_{i_2}(t)+x_3 W_{i_3}(t)+x_4 W_{i_4}(t))$ are 
denoted by
$$C(V,H,t)=[x_1x_2x_3x_4]
\det(x_1 V_{i_1}(t)+x_2 V_{i_2}(t)+x_3 V_{i_3}(t)+x_4 V_{i_4}(t))$$
and
$$C(W,H,t)=[x_1x_2x_3x_4]
\det(x_1 W_{i_1}(t)+x_2 W_{i_2}(t)+x_3 W_{i_3}(t)+x_4 W_{i_4}(t)),$$
respectively.
Let
$$C(H,t)=C(U,H,t)C(V,H,t)C(W,H,t)$$
denote their product, which is also a one-variable rational function in $t$.
Then define
\begin{equation}\label{eq-invdps}
\varphi (s(t))=\sum_{H\in R_2^4} C(H,t).
\end{equation}

For the four types of group actions of $\mathcal{I}(4,4,4|48)$, the de Groote, layer scaling and permutation summands actions leave the rank triples in (\ref{eq-rtriple}) invariant. Only the $S_3$-symmetry group can change the rank triples. However, by (\ref{eq-triple1}) and (\ref{eq-triple2}) we have that the rank-one  and rank-two triples are invariant under the $S_3$-symmetry group. Then by properties of the determinant, it is not hard to verify that  $\varphi (s(t))$ defined in (\ref{eq-invdps}) is an invariant function for the solution $s(t)$ in Appendix \ref{sec-app} with respect to the $\mathcal{I}(4,4,4|48)$-action.
After computation, we have
\begin{align*}
   \varphi (s(t))&=\sum_{H\in R_2^4} C(H,t)=64+\frac{4}{t^2}+1024 t^2  \\
                 &= 64\left(1+\frac{1}{16t^2}+16t^2\right).
\end{align*}
Thus, $\varphi (s(t))$ is an even function.
It is strictly decreasing  on each of the intervals
$$(-\infty, -\frac{1}{4}]\quad\text{and}\quad(0,\frac{1}{4}],$$
and strictly increasing on each of the intervals
$$[-\frac{1}{4},0)\quad\text{and}\quad[\frac{1}{4},+\infty).$$

Then in each monotone interval, if $t_1\neq t_2$, then $\varphi (s(t_1))\neq\varphi (s(t_2))$ and therefore
$s(t_1)$ and $s(t_2)$ belong to distinct $\mathcal{I}(4,4,4|48)$-orbits.
For example, all $s(t)$ belong to distinct $\mathcal{I}(4,4,4|48)$-orbits  when  $t\in[\frac{1}{4},+\infty)$. In particular, since the coefficients of $s(t)$ are rational, for any two rational numbers $t_1$ and $t_2$ in $[\frac{1}{4},+\infty)$, $s(t_1)$ and $s(t_2)$ are inequivalent 48-multiplication algorithms with rational coefficients.

Moreover, the index sets that give other rank gaps can also be found. For example, we can find $I$ with $|I|=2020$ such that
\begin{equation*} 
2=\operatorname{rank}N(s)-\operatorname{rank}N_I(s)>\operatorname{rank}T(s)-\operatorname{rank}T_I(s)=0.
\end{equation*}

\subsubsection{The test on AlphaTensor's solution in $\mathrm{V}(4,4,4|49)$}
\

In this section, the numerical tests are based on the solution in $\mathrm{V}(4,4,4|49)$ obtained by AlphaTensor \cite{Faw22}. 

When $m=n=p=4$ and $r=49$, the Brent equations $B(4,4,4|49)$ has 
$2352$ variables and $4096$ equations. From the deflation test in \cite{li25def}, we know that the first 4 numbers of deflation sequence of most solutions are
$$(197,197,197,197).$$
The constant deflation sequence provides numerical evidence that
the solution is smooth with local dimension 197.
We randomly choose one such solution and denote it by $s$.
After computing the corresponding Jacobian matrix $J(s)$ defined in (\ref{eq-js}), using the Julia package "Nemo.jl" we obtain the nullspace basis matrix $N(s)$, whose rank is $197$. After computing the tangent basis matrix $T(s)$ defined in (\ref{eq-tsmnpr}), we find that $\operatorname{rank}T(s)=143$.

In this case, it is easy to find an index set $I$ with $|I|=2344$, which gives a rank gap
\begin{equation*} 
1=\operatorname{rank}N(s)-\operatorname{rank}N_I(s)>\operatorname{rank}T(s)-\operatorname{rank}T_I(s)=0.
\end{equation*}
Since the number of variables becomes 
$$2352-2344=8,$$
the reduced polynomial system can be solved by hand. After solving the reduced polynomial system, we obtain a one-dimensional solution curve $s(t)$. We can check that the parameterized solution $s(t)$ satisfies (\ref{eq-tstmnp}) and (\ref{eq-tsst1}). Then $s(t)$ contains infinitely many pairwise inequivalent points.

Since the global dimension gap
$$\operatorname{rank}N(s)-\operatorname{rank}T(s)=197-143=54,$$
is large, we can obtain large rank gaps. For example, we can find $I$ with $|I|=2269$ such that
\begin{equation}\label{eq-gap10}
10=\operatorname{rank}N(s)-\operatorname{rank}N_I(s)>\operatorname{rank}T(s)-\operatorname{rank}T_I(s)=0.
\end{equation}
In the reduced polynomial system, we only have 83 variables. So in this case, we can obtain higher-dimensional parameterized solution sets. After fixing the partial solution $s_I$, the reduced polynomial system can often be divided into smaller individual subsystems. Each subsystem can be solved by computing the Gröbner basis. The dimension of the solution set of the reduced polynomial system is 9, which is less than the rank gap in (\ref{eq-gap10}). So the dimensions of the parameterized solution sets are not always equal to the rank gaps. The parameterized solution set we obtained can be found in \cite{L-23}.

\subsubsection{The test on Laderman's solution in $\mathrm{V}(3,3,3|23)$}
\

In this section, the numerical tests are based on the solution in $\mathrm{V}(3,3,3|23)$ obtained by Laderman \cite{Lad-76}. 

When $m=n=p=3$ and $r=23$, the Brent equations $B(3,3,3|23)$ has 
$621$ variables and $729$ equations.
From the deflation test in \cite{li25def}, we know that the first 4 numbers of deflation sequence of Laderman's solutions are
$$(76,76,76,76).$$
After computing the corresponding Jacobian matrix, we obtain the nullspace basis matrix $N(s)$, whose rank is $76$. For the tangent basis matrix $T(s)$, we find that $\operatorname{rank}T(s)=70$.

It is also not hard to find an index set $I$ with $|I|=613$, which gives a rank gap
\begin{equation*} 
1=\operatorname{rank}N(s)-\operatorname{rank}N_I(s)>\operatorname{rank}T(s)-\operatorname{rank}T_I(s)=0.
\end{equation*}
Since the number of variables is 8, the reduced polynomial system can be easily solved. After solving the reduced polynomial system, we obtain a one-dimensional parameterized solution curve $s(t)$. We can check that the parameterized solution $s(t)$ satisfies (\ref{eq-tstmnp}) and (\ref{eq-tsst1}). So $s(t)$ contains infinitely many pairwise inequivalent points.

We can also find other rank gaps. For example, there exists $I$ with $|I|=574$ such that
\begin{equation*}
 6=\operatorname{rank}N(s)-\operatorname{rank}N_I(s)>\operatorname{rank}T(s)-\operatorname{rank}T_I(s)=0.
\end{equation*}
In the reduced polynomial system, we only have 47 variables. The reduced Brent equations can be solved and the solution set is six-dimensional.

\begin{remark}
By Theorems 5 and 6 of \cite[Chapter 4  \S 5]{cox25}, we can see that a rational parameterized solution set belongs to an irreducible affine variety.
\end{remark}

\subsection{Numerical tests on singular solutions}\label{subs-testsing}
\

The deflation sequences in \cite{li25def} show that some solutions in 
$\mathrm{V}(m,n,p|r)$ are singular. In this section, we show our numerical tests on them. Our tests show that the substitution procedure in Section \ref{subs-subsprocedure} is also suitable to singular solutions.

Firstly, we choose a singular solution $s\in\mathrm{V}(4,4,4|49)$ found by AlphaTensor \cite{Faw22}. The first 4 numbers of deflation sequence of $s$ are $$(198,~182,~181,~181).$$
Computing the corresponding Jacobian matrix, we obtain the nullspace basis matrix $N(s)$, whose rank is $198$. For the tangent basis matrix $T(s)$, we find that $\operatorname{rank}T(s)= 143$.

It is also not hard to find an index set $I$ with $|I|=2330$, which gives a rank gap
\begin{equation*} 
2=\operatorname{rank}N(s)-\operatorname{rank}N_I(s)>
\operatorname{rank}T(s)-\operatorname{rank}T_I(s)=0.
\end{equation*}
After substitution, the number of variables in the reduced polynomial system is 22. Solving the reduced polynomial system, we obtain a two-dimensional parameterized solution set.
We can also find other rank gaps. For example, there exists $I$ with $|I|=2168$ such that
\begin{equation*}
16=\operatorname{rank}N(s)-\operatorname{rank}N_I(s)>
\operatorname{rank}T(s)-\operatorname{rank}T_I(s)=0.
\end{equation*}
The reduced polynomial system contains 184 variables. After solving the reduced Brent equations, a parameterized solution set with dimension 12 is obtained.

Second, we choose a singular solution $s\in\mathrm{V}(3,3,3|23)$ found by Smirnov \cite{Smir-13}. The first 4 numbers of deflation sequence of $s$ are
$$(87,~82,~79,~79).$$
After computing the corresponding Jacobian matrix, we obtain the nullspace basis matrix $N(s)$ with rank $87$. For the tangent basis matrix $T(s)$, we find that $\operatorname{rank}T(s)=70$.

It is also not hard to find an index set $I$ with $|I|=603$, which gives a rank gap
\begin{equation*} 
2=\operatorname{rank}N(s)-\operatorname{rank}N_I(s)>
\operatorname{rank}T(s)-\operatorname{rank}T_I(s)=0.
\end{equation*}
The reduced polynomial system contains 18 variables. Solving the reduced polynomial system, we obtain a two-dimensional parameterized solution set.
We can also find other rank gaps. For example, there exists $I$ with $|I|=550$ such that
\begin{equation*}
15=\operatorname{rank}N(s)-\operatorname{rank}N_I(s)>
\operatorname{rank}T(s)-\operatorname{rank}T_I(s)=0.
\end{equation*}
The reduced polynomial system contains 71 variables. The reduced Brent equations can be solved and the dimension of the solution set is 9.

\section{Final remarks and open problems}
\label{sec-remark}
In this section, we give some remarks and open problems to our substitution method.  
\subsection{Solvability and other computational methods}
\

In this paper, the parameterized solution sets of the reduced Brent equations are obtained from computation of the Gröbner bases \cite{cox25,Dec25}.  

Let $k=(mn+np+pm)r$, and let $s\in\mathrm{V}(m,n,p | r)$ be a solution. Suppose that an index set $I\subseteq [k]$ is obtained by \textbf{Step} 5 of Section \ref{subs-subsprocedure}. Then number of variables in the reduced Brent equations is $$a=k-|I|.$$
The solvability of the reduced Brent equations can be gauged by $a$. From our numerical tests we find that if $a\leq 200$, the Gröbner bases of the reduced Brent equations are easy to be found; if $200<a\leq 300$, 
the Gröbner bases can still be obtained for some reduced Brent equations.

It is well known that the computation of the Gröbner basis is expensive. However,  after making a substitution, the complexity of the reduced Brent equations is much lower than before. For example, we can find some candidate index sets $I$ such that $a\leq 600$.  If the Gröbner bases are hard to be computed, then other classical computational methods can be used to find solutions of the reduced Brent equations, such as the nonlinear least-squares method \cite{kaporin24-cmmp,kaporin24-Dokl,CV25} and the SAT-based method \cite{Heule21}.

\subsection{The local coordinate cross-section and small index set}
\

Every tensor admits a layer scaling action. When finding other tensor decompositions from a known decomposition, the quotient of the layer scaling action is easy to be realized in a minimal way. For example, following the method in \cite{HOOS2019} we can fix any two nonzero partial solutions in each rank-one summand.
However, when there exist other nontrivial positive-dimensional isotropy group actions, such as the de Groote action, are there some methods to realize the quotient by fixing the minimal number of the partial solutions?

Let $T(s)$ be the tangent basis matrix defined in \ref{eq-Ts}. Suppose that
$\operatorname{rank}T(s)=q$. Then the quotient of the positive-dimensional isotropy group action can be realized by choosing an index set $I$ with $|I|=q$. In fact, since 
$\operatorname{rank}T(s)=q$, there exist $q$ rows of $T(s)$ that are linear independent. Let $I$ be the corresponding row index set. Then we have $|I|=q$ and
$$\operatorname{rank}T_I(s)=q.$$
It is not hard to see that $q$ is the minimal number of partial solutions that should be fixed.

Take the Brent equations as an example. For the solution $s\in \mathrm{V}(4,4,4|48)$ found in \cite{dumas25}, by the discussion in Section \ref{subs-testdps}, we know that $\operatorname{rank}N(s)=151$ and $\operatorname{rank}T(s)=141$. So we can find an index set $I$ such that $|I|=141$ and $\operatorname{rank}T_I(s)=141$. By Lemma \ref{lem-subfull} we also have $\operatorname{rank}N_I(s)=141$. Then the rank gap associated to $I$ is $10$. Thus, the de Groote and layer scaling actions are quotiented in a minimal way. Moreover, by fixing partial solution that indexed by $I$, we obtain a local (coordinate) cross-section \cite{Hubert07sa,olver12mf}.

\subsection{On searching the index sets}
\

To find candidate index sets in \textbf{Step} 5 of Section \ref{subs-subsprocedure}, we use the method raised in Section \ref{subs-findI}. Although many candidate index sets can be obtained,  
we cannot guarantee the index sets are optimal. Therefore, by improving the method in Section \ref{subs-findI}, more computable reduced Brent equations can be found. Nowadays, $I$ can also be found by the Large Language Model.

\subsection{Finding integer solutions in $\mathrm{V}(4,4,4|48)$}
\

In \cite{moran26}, the authors discuss how to transform a complex-valued solution of $\mathrm{V}(m,n,p|r)$ into a rational solution. They also give some necessary conditions that decide if a solution of $\mathrm{V}(m,n,p|r)$ can be transformed into an integer solution. They show that the solution
of \cite{dumas25} cannot be transformed into an integer solution. However, after making the substitution as in Section \ref{subs-testdps} we can try to find integer solutions in the parameterized solution sets.

\subsection{Equivalence of the parameterized solution sets}
\

Under the isotropy group action, it is interesting to decide when two parameterized solution sets are equivalent.  As an example, consider two one-dimensional parameterized solution sets $s_1(t)$ and $s_2(t)$ in $\mathrm{V}(m,n,p|r)$. We want to know when there exists an element $g\in \mathcal{I}(m,n,p|r)$ such that $s_1(t)=gs_2(t)$.

\section*{Acknowledgments}
Many thanks to Zehao Quan and Zhe Xie for their helpful discussions. We are grateful to Professor Ke Ye for his valuable suggestions.

\bibliographystyle{amsplain}

\appendix
\section{A Parameterized Rational Solution in $\mathrm{V}(4,4,4|48)$}
\label{sec-app}

For every $t\in\mathbb{R}^{\times}$, the matrices below satisfy
$$
\sum_{i=1}^{48} U_i(t)\otimes V_i(t)\otimes W_i(t)
=\langle4,4,4\rangle.$$ 
This family is obtained by parameterizing the solution provided in
\cite{dumas25}.

\begingroup
\begin{longtable}{|c|c|c|c|}
\hline
 & $U_i(t)$ & $V_i(t)$ & $W_i(t)$ \\ \hline
\endfirsthead
\hline
 & $U_i(t)$ & $V_i(t)$ & $W_i(t)$ \\ \hline
\endhead
\hline
\endfoot
\hline
\endlastfoot
$i=1$
& $\begin{matrix}
-1 & 1 & 1 & 1 \\
-1 & 1 & 1 & 1 \\
1 & -1 & -1 & -1 \\
-1 & 1 & 1 & 1
\end{matrix}$
& $\begin{matrix}
0 & 0 & 0 & 0 \\
0 & 0 & 0 & 0 \\
1 & 0 & 0 & 0 \\
1 & 0 & 0 & 0
\end{matrix}$
& $\begin{matrix}
0 & 0 & -\tfrac{1}{4} & \tfrac{1}{4} \\
0 & 0 & \tfrac{1}{4} & -\tfrac{1}{4} \\
0 & 0 & -\tfrac{4t-1}{4} & \tfrac{4t-1}{4} \\
0 & 0 & -\tfrac{1}{4} & \tfrac{1}{4}
\end{matrix}$ \\ \hline
$i=2$
& $\begin{matrix}
-1 & 0 & 0 & 0 \\
1 & 0 & 0 & 0 \\
1 & 0 & 0 & 0 \\
-1 & 0 & 0 & 0
\end{matrix}$
& $\begin{matrix}
0 & 1 & 0 & 1 \\
0 & -1 & 0 & -1 \\
0 & 0 & 0 & 0 \\
0 & 0 & 0 & 0
\end{matrix}$
& $\begin{matrix}
-\tfrac{1}{8t} & \tfrac{1}{8t} & 0 & 0 \\
0 & 0 & 0 & 0 \\
0 & 0 & 0 & 0 \\
-\tfrac{1}{2} & \tfrac{1}{2} & 0 & 0
\end{matrix}$ \\ \hline
$i=3$
& $\begin{matrix}
0 & 0 & -1 & 0 \\
0 & 0 & 1 & 0 \\
0 & 0 & 1 & 0 \\
0 & 0 & 1 & 0
\end{matrix}$
& $\begin{matrix}
0 & 1 & 0 & 0 \\
0 & 0 & 0 & 0 \\
0 & 1 & 0 & 0 \\
0 & 0 & 0 & 0
\end{matrix}$
& $\begin{matrix}
0 & 0 & 0 & 0 \\
0 & \tfrac{1}{2} & 0 & \tfrac{1}{2} \\
0 & -2t & 0 & -2t \\
0 & \tfrac{4t-1}{2} & 0 & \tfrac{4t-1}{2}
\end{matrix}$ \\ \hline
$i=4$
& $\begin{matrix}
0 & 0 & -1 & -1 \\
0 & 0 & -1 & -1 \\
0 & 0 & 1 & -1 \\
0 & 0 & 1 & -1
\end{matrix}$
& $\begin{matrix}
0 & 0 & 0 & 0 \\
0 & 0 & 0 & 0 \\
0 & t & \tfrac{1}{4} & 0 \\
t & 0 & -\tfrac{1}{4} & 0
\end{matrix}$
& $\begin{matrix}
0 & 0 & 0 & 0 \\
-\tfrac{1}{2t} & -\tfrac{1}{2t} & 0 & 0 \\
1 & 1 & 1 & 1 \\
-\tfrac{2t-1}{2t} & -\tfrac{2t-1}{2t} & -1 & -1
\end{matrix}$ \\ \hline
$i=5$
& $\begin{matrix}
1 & 1 & 1 & -1 \\
-1 & -1 & -1 & 1 \\
1 & 1 & 1 & -1 \\
1 & 1 & 1 & -1
\end{matrix}$
& $\begin{matrix}
-4t & 0 & 1 & 1 \\
-4t & 0 & 1 & 1 \\
0 & 0 & 0 & 0 \\
0 & 0 & 0 & 0
\end{matrix}$
& $\begin{matrix}
0 & 0 & 0 & 0 \\
0 & 0 & 0 & 0 \\
0 & 0 & 0 & 0 \\
0 & 0 & \tfrac{1}{4} & \tfrac{1}{4}
\end{matrix}$ \\ \hline
$i=6$
& $\begin{matrix}
1 & 1 & -1 & 1 \\
-1 & -1 & 1 & -1 \\
-1 & -1 & 1 & -1 \\
1 & 1 & -1 & 1
\end{matrix}$
& $\begin{matrix}
0 & 0 & 0 & 0 \\
0 & 1 & 0 & 1 \\
0 & 0 & 0 & 0 \\
0 & 1 & 0 & 1
\end{matrix}$
& $\begin{matrix}
0 & \tfrac{1}{4} & 0 & -\tfrac{1}{4} \\
0 & -\tfrac{1}{4} & 0 & \tfrac{1}{4} \\
0 & t & 0 & -t \\
0 & 0 & 0 & 0
\end{matrix}$ \\ \hline
$i=7$
& $\begin{matrix}
0 & 0 & 0 & -1 \\
0 & 0 & 0 & 1 \\
0 & 0 & 0 & 1 \\
0 & 0 & 0 & 1
\end{matrix}$
& $\begin{matrix}
0 & 0 & 0 & 0 \\
0 & 1 & 0 & 0 \\
0 & 0 & 0 & 0 \\
0 & -1 & 0 & 0
\end{matrix}$
& $\begin{matrix}
0 & 0 & 0 & 0 \\
\tfrac{1}{2} & 0 & -\tfrac{1}{2} & 0 \\
0 & 0 & 0 & 0 \\
-\tfrac{1}{2} & 0 & \tfrac{1}{2} & 0
\end{matrix}$ \\ \hline
$i=8$
& $\begin{matrix}
1 & 1 & 1 & -1 \\
-1 & -1 & -1 & 1 \\
1 & 1 & 1 & -1 \\
-1 & -1 & -1 & 1
\end{matrix}$
& $\begin{matrix}
4t & 4t-1 & 0 & -1 \\
0 & 0 & 0 & 0 \\
4t & 4t-1 & 0 & -1 \\
0 & 0 & 0 & 0
\end{matrix}$
& $\begin{matrix}
0 & 0 & 0 & 0 \\
0 & 0 & 0 & 0 \\
0 & 0 & 0 & 0 \\
0 & \tfrac{1}{4} & 0 & \tfrac{1}{4}
\end{matrix}$ \\ \hline
$i=9$
& $\begin{matrix}
0 & 0 & -1 & -1 \\
0 & 0 & 1 & 1 \\
-1 & 1 & 0 & 0 \\
-1 & 1 & 0 & 0
\end{matrix}$
& $\begin{matrix}
4t & 1 & -1 & -1 \\
-4t & -1 & 1 & 1 \\
-4t & 1 & 1 & 1 \\
-4t & 1 & 1 & 1
\end{matrix}$
& $\begin{matrix}
\tfrac{1}{8} & -\tfrac{1}{8} & \tfrac{1}{8} & \tfrac{1}{8} \\
-\tfrac{1}{8} & \tfrac{1}{8} & -\tfrac{1}{8} & -\tfrac{1}{8} \\
\tfrac{2t-1}{4} & -\tfrac{2t-1}{4} & \tfrac{t}{2} & \tfrac{t}{2} \\
\tfrac{1}{8} & -\tfrac{1}{8} & \tfrac{1}{8} & \tfrac{1}{8}
\end{matrix}$ \\ \hline
$i=10$
& $\begin{matrix}
1 & 0 & 0 & 0 \\
1 & 0 & 0 & 0 \\
1 & 0 & 0 & 0 \\
-1 & 0 & 0 & 0
\end{matrix}$
& $\begin{matrix}
1-4t & 1 & 1 & 1 \\
0 & 0 & 0 & 0 \\
4t-1 & -1 & -1 & -1 \\
0 & 0 & 0 & 0
\end{matrix}$
& $\begin{matrix}
\tfrac{1}{2} & 0 & \tfrac{1}{2} & 0 \\
0 & 0 & 0 & 0 \\
0 & 0 & 0 & 0 \\
2t & 0 & 2t & 0
\end{matrix}$ \\ \hline
$i=11$
& $\begin{matrix}
-1 & -1 & -1 & 1 \\
1 & 1 & 1 & -1 \\
1 & 1 & 1 & -1 \\
-1 & -1 & -1 & 1
\end{matrix}$
& $\begin{matrix}
0 & 0 & 0 & 0 \\
0 & 1 & 0 & 1 \\
0 & 0 & 0 & 0 \\
0 & -1 & 0 & -1
\end{matrix}$
& $\begin{matrix}
0 & 0 & 0 & 0 \\
0 & 0 & 0 & 0 \\
0 & 0 & 0 & 0 \\
-\tfrac{1}{4} & 0 & \tfrac{1}{4} & 0
\end{matrix}$ \\ \hline
$i=12$
& $\begin{matrix}
0 & 1 & 0 & 1 \\
-1 & 0 & 1 & 0 \\
0 & 1 & 0 & 1 \\
-1 & 0 & 1 & 0
\end{matrix}$
& $\begin{matrix}
-4t & 1-8t & -1 & 1 \\
4t & -1 & -1 & -1 \\
4t & 8t-1 & 1 & -1 \\
4t & -1 & -1 & -1
\end{matrix}$
& $\begin{matrix}
\tfrac{1}{4} & 0 & \tfrac{1}{4} & 0 \\
-\tfrac{1}{4} & 0 & -\tfrac{1}{4} & 0 \\
\tfrac{8t-1}{8} & \tfrac{1}{8} & \tfrac{8t-1}{8} & \tfrac{1}{8} \\
\tfrac{1}{8} & -\tfrac{1}{8} & \tfrac{1}{8} & -\tfrac{1}{8}
\end{matrix}$ \\ \hline
$i=13$
& $\begin{matrix}
0 & 0 & 1 & 0 \\
0 & 0 & -1 & 0 \\
0 & 0 & 1 & 0 \\
0 & 0 & 1 & 0
\end{matrix}$
& $\begin{matrix}
4t & 0 & -1 & -1 \\
0 & 0 & 0 & 0 \\
-4t & 0 & 1 & 1 \\
0 & 0 & 0 & 0
\end{matrix}$
& $\begin{matrix}
0 & 0 & 0 & 0 \\
-\tfrac{1}{2} & 0 & -\tfrac{1}{2} & 0 \\
2t & 0 & 2t & 0 \\
-\tfrac{4t-1}{2} & 0 & -\tfrac{4t-1}{2} & 0
\end{matrix}$ \\ \hline
$i=14$
& $\begin{matrix}
0 & 0 & 1 & 0 \\
0 & 0 & 1 & 0 \\
0 & 0 & -1 & 0 \\
0 & 0 & -1 & 0
\end{matrix}$
& $\begin{matrix}
0 & 0 & 0 & 0 \\
0 & 0 & 0 & 0 \\
4t & 0 & -1 & 0 \\
-4t & 0 & 1 & 0
\end{matrix}$
& $\begin{matrix}
0 & 0 & 0 & 0 \\
\tfrac{1}{8t} & \tfrac{1}{8t} & 0 & 0 \\
-\tfrac{1}{2} & -\tfrac{1}{2} & 0 & 0 \\
\tfrac{4t-1}{8t} & \tfrac{4t-1}{8t} & 0 & 0
\end{matrix}$ \\ \hline
$i=15$
& $\begin{matrix}
1 & 0 & 1 & 0 \\
1 & 0 & -1 & 0 \\
-1 & 0 & -1 & 0 \\
1 & 0 & -1 & 0
\end{matrix}$
& $\begin{matrix}
0 & 1 & 0 & 0 \\
0 & 0 & 0 & 0 \\
-1 & 0 & 0 & 0 \\
0 & 0 & 0 & 0
\end{matrix}$
& $\begin{matrix}
-\tfrac{1}{4} & \tfrac{1}{4} & \tfrac{1}{4} & \tfrac{1}{4} \\
\tfrac{1}{4} & \tfrac{1}{4} & -\tfrac{1}{4} & \tfrac{1}{4} \\
-t & -t & t & -t \\
-\tfrac{1}{4} & \tfrac{8t-1}{4} & \tfrac{1}{4} & \tfrac{8t-1}{4}
\end{matrix}$ \\ \hline
$i=16$
& $\begin{matrix}
1 & 1 & -1 & 1 \\
-1 & -1 & 1 & -1 \\
1 & 1 & -1 & 1 \\
-1 & -1 & 1 & -1
\end{matrix}$
& $\begin{matrix}
4t & 4t-1 & 0 & -1 \\
0 & 0 & 0 & 0 \\
-4t & 1-4t & 0 & 1 \\
0 & 0 & 0 & 0
\end{matrix}$
& $\begin{matrix}
\tfrac{1}{4} & 0 & \tfrac{1}{4} & 0 \\
-\tfrac{1}{4} & 0 & -\tfrac{1}{4} & 0 \\
t & 0 & t & 0 \\
0 & 0 & 0 & 0
\end{matrix}$ \\ \hline
$i=17$
& $\begin{matrix}
1 & 1 & 0 & 0 \\
1 & 1 & 0 & 0 \\
-1 & 1 & 0 & 0 \\
-1 & 1 & 0 & 0
\end{matrix}$
& $\begin{matrix}
0 & -4t & -1 & 0 \\
4t & 0 & -1 & 0 \\
0 & 0 & 0 & 0 \\
0 & 0 & 0 & 0
\end{matrix}$
& $\begin{matrix}
0 & 0 & \tfrac{1}{8t} & \tfrac{1}{8t} \\
0 & 0 & 0 & 0 \\
-\tfrac{1}{4} & -\tfrac{1}{4} & \tfrac{1}{4} & \tfrac{1}{4} \\
\tfrac{1}{4} & \tfrac{1}{4} & \tfrac{1}{4} & \tfrac{1}{4}
\end{matrix}$ \\ \hline
$i=18$
& $\begin{matrix}
-1 & 1 & 1 & 1 \\
-1 & 1 & 1 & 1 \\
1 & -1 & -1 & -1 \\
1 & -1 & -1 & -1
\end{matrix}$
& $\begin{matrix}
0 & 0 & 0 & 0 \\
4t & 0 & -1 & 0 \\
0 & 0 & 0 & 0 \\
4t & 0 & -1 & 0
\end{matrix}$
& $\begin{matrix}
0 & \tfrac{1}{4} & 0 & -\tfrac{1}{4} \\
0 & -\tfrac{1}{4} & 0 & \tfrac{1}{4} \\
0 & \tfrac{4t-1}{4} & 0 & -\tfrac{4t-1}{4} \\
0 & \tfrac{1}{4} & 0 & -\tfrac{1}{4}
\end{matrix}$ \\ \hline
$i=19$
& $\begin{matrix}
0 & 1 & 0 & 1 \\
-1 & 0 & 1 & 0 \\
0 & -1 & 0 & -1 \\
1 & 0 & -1 & 0
\end{matrix}$
& $\begin{matrix}
4t & -1 & -1 & -1 \\
4t & 1 & -1 & 1 \\
-4t & 1 & 1 & 1 \\
4t & 1 & -1 & 1
\end{matrix}$
& $\begin{matrix}
0 & -\tfrac{1}{4} & 0 & \tfrac{1}{4} \\
0 & \tfrac{1}{4} & 0 & -\tfrac{1}{4} \\
-\tfrac{1}{8} & -\tfrac{8t-1}{8} & \tfrac{1}{8} & \tfrac{8t-1}{8} \\
\tfrac{1}{8} & -\tfrac{1}{8} & -\tfrac{1}{8} & \tfrac{1}{8}
\end{matrix}$ \\ \hline
$i=20$
& $\begin{matrix}
-1 & 1 & -1 & -1 \\
-1 & 1 & -1 & -1 \\
1 & -1 & 1 & 1 \\
1 & -1 & 1 & 1
\end{matrix}$
& $\begin{matrix}
0 & 0 & 0 & 0 \\
-4t & 0 & 1 & 0 \\
0 & 0 & 0 & 0 \\
4t & 0 & -1 & 0
\end{matrix}$
& $\begin{matrix}
0 & 0 & 0 & 0 \\
0 & 0 & 0 & 0 \\
\tfrac{1}{4} & 0 & -\tfrac{1}{4} & 0 \\
0 & 0 & 0 & 0
\end{matrix}$ \\ \hline
$i=21$
& $\begin{matrix}
0 & 0 & 1 & 0 \\
0 & 0 & -1 & 0 \\
0 & 0 & -1 & 0 \\
0 & 0 & 1 & 0
\end{matrix}$
& $\begin{matrix}
0 & 0 & 0 & 0 \\
0 & 0 & 0 & 0 \\
0 & 1 & 0 & 1 \\
0 & 1 & 0 & 1
\end{matrix}$
& $\begin{matrix}
0 & 0 & 0 & 0 \\
0 & 0 & -\tfrac{1}{8t} & \tfrac{1}{8t} \\
0 & 0 & \tfrac{1}{2} & -\tfrac{1}{2} \\
0 & 0 & -\tfrac{4t-1}{8t} & \tfrac{4t-1}{8t}
\end{matrix}$ \\ \hline
$i=22$
& $\begin{matrix}
1 & 0 & 1 & 0 \\
0 & -1 & 0 & 1 \\
1 & 0 & 1 & 0 \\
0 & -1 & 0 & 1
\end{matrix}$
& $\begin{matrix}
4t & 8t-1 & 1 & -1 \\
-4t & 1 & 1 & 1 \\
4t & 8t-1 & 1 & -1 \\
4t & -1 & -1 & -1
\end{matrix}$
& $\begin{matrix}
0 & 0 & 0 & 0 \\
0 & 0 & 0 & 0 \\
\tfrac{1}{8} & -\tfrac{1}{8} & \tfrac{1}{8} & -\tfrac{1}{8} \\
-\tfrac{1}{8} & -\tfrac{1}{8} & -\tfrac{1}{8} & -\tfrac{1}{8}
\end{matrix}$ \\ \hline
$i=23$
& $\begin{matrix}
0 & 0 & 0 & 1 \\
0 & 0 & 0 & -1 \\
0 & 0 & 0 & 1 \\
0 & 0 & 0 & 1
\end{matrix}$
& $\begin{matrix}
0 & 0 & 0 & 0 \\
-4t & 0 & 1 & 1 \\
0 & 0 & 0 & 0 \\
-4t & 0 & 1 & 1
\end{matrix}$
& $\begin{matrix}
0 & 0 & 0 & 0 \\
0 & \tfrac{1}{2} & 0 & -\tfrac{1}{2} \\
0 & 0 & 0 & 0 \\
0 & -\tfrac{1}{2} & 0 & \tfrac{1}{2}
\end{matrix}$ \\ \hline
$i=24$
& $\begin{matrix}
1 & -1 & 1 & 1 \\
-1 & 1 & -1 & -1 \\
1 & -1 & 1 & 1 \\
1 & -1 & 1 & 1
\end{matrix}$
& $\begin{matrix}
4t & 0 & -1 & -1 \\
-4t & 0 & 1 & 1 \\
0 & 0 & 0 & 0 \\
0 & 0 & 0 & 0
\end{matrix}$
& $\begin{matrix}
0 & 0 & 0 & 0 \\
0 & 0 & 0 & 0 \\
-\tfrac{1}{4} & \tfrac{1}{4} & 0 & 0 \\
0 & 0 & 0 & 0
\end{matrix}$ \\ \hline
$i=25$
& $\begin{matrix}
1 & 1 & 0 & 0 \\
-1 & -1 & 0 & 0 \\
0 & 0 & 1 & -1 \\
0 & 0 & 1 & -1
\end{matrix}$
& $\begin{matrix}
4t & 1 & -1 & -1 \\
4t & 1 & -1 & -1 \\
-4t & 1 & 1 & 1 \\
4t & -1 & -1 & -1
\end{matrix}$
& $\begin{matrix}
-\tfrac{1}{8} & \tfrac{1}{8} & -\tfrac{1}{8} & -\tfrac{1}{8} \\
\tfrac{1}{8} & -\tfrac{1}{8} & \tfrac{1}{8} & \tfrac{1}{8} \\
-\tfrac{t}{2} & \tfrac{t}{2} & -\tfrac{t}{2} & -\tfrac{t}{2} \\
-\tfrac{1}{8} & \tfrac{1}{8} & \tfrac{1}{8} & \tfrac{1}{8}
\end{matrix}$ \\ \hline
$i=26$
& $\begin{matrix}
1 & -1 & 1 & 1 \\
1 & -1 & 1 & 1 \\
1 & -1 & 1 & 1 \\
-1 & 1 & -1 & -1
\end{matrix}$
& $\begin{matrix}
0 & 0 & 0 & 0 \\
0 & 0 & 0 & 0 \\
1-4t & 1 & 1 & 1 \\
1-4t & 1 & 1 & 1
\end{matrix}$
& $\begin{matrix}
0 & 0 & 0 & 0 \\
0 & 0 & 0 & 0 \\
0 & 0 & \tfrac{1}{4} & -\tfrac{1}{4} \\
0 & 0 & 0 & 0
\end{matrix}$ \\ \hline
$i=27$
& $\begin{matrix}
0 & 1 & 0 & 0 \\
0 & 1 & 0 & 0 \\
0 & 1 & 0 & 0 \\
0 & 1 & 0 & 0
\end{matrix}$
& $\begin{matrix}
0 & 1 & \tfrac{1}{4t} & 0 \\
0 & 1 & \tfrac{1}{4t} & 0 \\
0 & 0 & 0 & 0 \\
0 & 0 & 0 & 0
\end{matrix}$
& $\begin{matrix}
0 & 0 & \tfrac{1}{2} & \tfrac{1}{2} \\
0 & 0 & 0 & 0 \\
0 & 0 & 2t & 2t \\
0 & 0 & 0 & 0
\end{matrix}$ \\ \hline
$i=28$
& $\begin{matrix}
0 & 0 & 1 & 1 \\
0 & 0 & 1 & 1 \\
1 & -1 & 0 & 0 \\
-1 & 1 & 0 & 0
\end{matrix}$
& $\begin{matrix}
4t & -1 & -1 & -1 \\
-4t & 1 & 1 & 1 \\
-4t+2& 1 & 1 & 1 \\
-4t+2 & 1 & 1 & 1
\end{matrix}$
& $\begin{matrix}
\tfrac{1}{8} & \tfrac{1}{8} & \tfrac{1}{8} & -\tfrac{1}{8} \\
-\tfrac{1}{8} & -\tfrac{1}{8} & -\tfrac{1}{8} & \tfrac{1}{8} \\
\tfrac{t}{2} & \tfrac{t}{2} & \tfrac{2t-1}{4} & -\tfrac{2t-1}{4} \\
\tfrac{1}{8} & \tfrac{1}{8} & \tfrac{1}{8} & -\tfrac{1}{8}
\end{matrix}$ \\ \hline
$i=29$
& $\begin{matrix}
1 & -1 & -1 & -1 \\
-1 & 1 & 1 & 1 \\
-1 & 1 & 1 & 1 \\
-1 & 1 & 1 & 1
\end{matrix}$
& $\begin{matrix}
0 & -1 & 0 & 0 \\
0 & 1 & 0 & 0 \\
0 & 0 & 0 & 0 \\
0 & 0 & 0 & 0
\end{matrix}$
& $\begin{matrix}
\tfrac{1}{4} & -\tfrac{1}{4} & 0 & 0 \\
-\tfrac{1}{4} & \tfrac{1}{4} & 0 & 0 \\
\tfrac{4t-1}{4} & -\tfrac{4t-1}{4} & 0 & 0 \\
\tfrac{1}{4} & -\tfrac{1}{4} & 0 & 0
\end{matrix}$ \\ \hline
$i=30$
& $\begin{matrix}
0 & 0 & 0 & 1 \\
0 & 0 & 0 & 1 \\
0 & 0 & 0 & 1 \\
0 & 0 & 0 & 1
\end{matrix}$
& $\begin{matrix}
0 & 0 & 0 & 0 \\
0 & 0 & 0 & 0 \\
0 & -1 & -\tfrac{1}{4t} & 0 \\
0 & 1 & \tfrac{1}{4t} & 0
\end{matrix}$
& $\begin{matrix}
0 & 0 & 0 & 0 \\
\tfrac{1}{2} & \tfrac{1}{2} & 0 & 0 \\
0 & 0 & 0 & 0 \\
-\tfrac{1}{2} & -\tfrac{1}{2} & 0 & 0
\end{matrix}$ \\ \hline
$i=31$
& $\begin{matrix}
0 & 1 & 0 & 0 \\
0 & -1 & 0 & 0 \\
0 & 1 & 0 & 0 \\
0 & -1 & 0 & 0
\end{matrix}$
& $\begin{matrix}
1 & \tfrac{4t-1}{4t} & 0 & -\tfrac{1}{4t} \\
-1 & -\tfrac{4t-1}{4t} & 0 & \tfrac{1}{4t} \\
0 & 0 & 0 & 0 \\
0 & 0 & 0 & 0
\end{matrix}$
& $\begin{matrix}
-\tfrac{1}{2} & \tfrac{1}{2} & 0 & 0 \\
0 & 0 & 0 & 0 \\
-2t & 2t & 0 & 0 \\
0 & 0 & 0 & 0
\end{matrix}$ \\ \hline
$i=32$
& $\begin{matrix}
1 & 0 & 1 & 0 \\
0 & -1 & 0 & 1 \\
-1 & 0 & -1 & 0 \\
0 & 1 & 0 & -1
\end{matrix}$
& $\begin{matrix}
4t & -1 & -1 & -1 \\
4t & 1 & -1 & 1 \\
4t & -1 & -1 & -1 \\
-4t & -1 & 1 & -1
\end{matrix}$
& $\begin{matrix}
0 & 0 & 0 & 0 \\
0 & 0 & 0 & 0 \\
-\tfrac{1}{8} & \tfrac{1}{8} & \tfrac{1}{8} & -\tfrac{1}{8} \\
-\tfrac{1}{8} & -\tfrac{1}{8} & \tfrac{1}{8} & \tfrac{1}{8}
\end{matrix}$ \\ \hline
$i=33$
& $\begin{matrix}
-1 & -1 & 1 & -1 \\
1 & 1 & -1 & 1 \\
1 & 1 & -1 & 1 \\
1 & 1 & -1 & 1
\end{matrix}$
& $\begin{matrix}
0 & -1 & 0 & 0 \\
0 & -1 & 0 & 0 \\
0 & 0 & 0 & 0 \\
0 & 0 & 0 & 0
\end{matrix}$
& $\begin{matrix}
0 & 0 & \tfrac{1}{4} & \tfrac{1}{4} \\
0 & 0 & -\tfrac{1}{4} & -\tfrac{1}{4} \\
0 & 0 & t & t \\
0 & 0 & 0 & 0
\end{matrix}$ \\ \hline
$i=34$
& $\begin{matrix}
0 & 0 & -1 & -1 \\
0 & 0 & 1 & 1 \\
0 & 0 & 1 & -1 \\
0 & 0 & -1 & 1
\end{matrix}$
& $\begin{matrix}
0 & 0 & 0 & 0 \\
0 & 0 & 0 & 0 \\
-4t & 1-4t & 0 & 1 \\
0 & 1 & 0 & 1
\end{matrix}$
& $\begin{matrix}
0 & 0 & 0 & 0 \\
0 & 0 & -\tfrac{1}{8t} & \tfrac{1}{8t} \\
\tfrac{1}{4} & -\tfrac{1}{4} & \tfrac{1}{4} & -\tfrac{1}{4} \\
-\tfrac{1}{4} & \tfrac{1}{4} & -\tfrac{2t-1}{8t} & \tfrac{2t-1}{8t}
\end{matrix}$ \\ \hline
$i=35$
& $\begin{matrix}
1 & 1 & 0 & 0 \\
1 & 1 & 0 & 0 \\
0 & 0 & 1 & -1 \\
0 & 0 & -1 & 1
\end{matrix}$
& $\begin{matrix}
-4t & 1 & 1 & 1 \\
-4t & 1 & 1 & 1 \\
4t-2 & -1 & -1 & -1 \\
-4t+2 & 1 & 1 & 1
\end{matrix}$
& $\begin{matrix}
-\tfrac{1}{8} & -\tfrac{1}{8} & -\tfrac{1}{8} & \tfrac{1}{8} \\
\tfrac{1}{8} & \tfrac{1}{8} & \tfrac{1}{8} & -\tfrac{1}{8} \\
-\tfrac{t}{2} & -\tfrac{t}{2} & -\tfrac{t}{2} & \tfrac{t}{2} \\
\tfrac{1}{8} & \tfrac{1}{8} & -\tfrac{1}{8} & \tfrac{1}{8}
\end{matrix}$ \\ \hline
$i=36$
& $\begin{matrix}
1 & -1 & 1 & 1 \\
1 & -1 & 1 & 1 \\
1 & -1 & 1 & 1 \\
1 & -1 & 1 & 1
\end{matrix}$
& $\begin{matrix}
0 & 4t & 1 & 0 \\
0 & 0 & 0 & 0 \\
0 & 4t & 1 & 0 \\
0 & 0 & 0 & 0
\end{matrix}$
& $\begin{matrix}
0 & 0 & 0 & 0 \\
0 & 0 & 0 & 0 \\
0 & \tfrac{1}{4} & 0 & \tfrac{1}{4} \\
0 & 0 & 0 & 0
\end{matrix}$ \\ \hline
$i=37$
& $\begin{matrix}
0 & 1 & 0 & 0 \\
0 & 1 & 0 & 0 \\
0 & 1 & 0 & 0 \\
0 & -1 & 0 & 0
\end{matrix}$
& $\begin{matrix}
0 & 0 & 0 & 0 \\
1-4t & 1 & 1 & 1 \\
0 & 0 & 0 & 0 \\
1-4t & 1 & 1 & 1
\end{matrix}$
& $\begin{matrix}
0 & \tfrac{1}{2} & 0 & -\tfrac{1}{2} \\
0 & 0 & 0 & 0 \\
0 & 2t & 0 & -2t \\
0 & 0 & 0 & 0
\end{matrix}$ \\ \hline
$i=38$
& $\begin{matrix}
0 & 1 & 0 & 1 \\
0 & 1 & 0 & -1 \\
0 & 1 & 0 & 1 \\
0 & -1 & 0 & 1
\end{matrix}$
& $\begin{matrix}
0 & 0 & 0 & 0 \\
-4t & 0 & 1 & 1 \\
0 & 0 & 0 & 0 \\
1-4t & 1 & 1 & 1
\end{matrix}$
& $\begin{matrix}
\tfrac{1}{4} & -\tfrac{1}{4} & \tfrac{1}{4} & \tfrac{1}{4} \\
-\tfrac{1}{4} & -\tfrac{1}{4} & -\tfrac{1}{4} & \tfrac{1}{4} \\
t & -t & t & t \\
\tfrac{1}{4} & \tfrac{1}{4} & \tfrac{1}{4} & -\tfrac{1}{4}
\end{matrix}$ \\ \hline
$i=39$
& $\begin{matrix}
0 & 1 & 0 & 0 \\
0 & 1 & 0 & 0 \\
0 & -1 & 0 & 0 \\
0 & 1 & 0 & 0
\end{matrix}$
& $\begin{matrix}
0 & 0 & 0 & 0 \\
-1 & 0 & 0 & 0 \\
0 & 0 & 0 & 0 \\
1 & 0 & 0 & 0
\end{matrix}$
& $\begin{matrix}
-\tfrac{1}{2} & 0 & \tfrac{1}{2} & 0 \\
0 & 0 & 0 & 0 \\
-2t & 0 & 2t & 0 \\
0 & 0 & 0 & 0
\end{matrix}$ \\ \hline
$i=40$
& $\begin{matrix}
-1 & -1 & 0 & 0 \\
1 & 1 & 0 & 0 \\
1 & -1 & 0 & 0 \\
-1 & 1 & 0 & 0
\end{matrix}$
& $\begin{matrix}
4t & 4t-1 & 0 & -1 \\
0 & 1 & 0 & 1 \\
0 & 0 & 0 & 0 \\
0 & 0 & 0 & 0
\end{matrix}$
& $\begin{matrix}
-\tfrac{1}{8t} & \tfrac{1}{8t} & 0 & 0 \\
0 & 0 & 0 & 0 \\
-\tfrac{1}{4} & \tfrac{1}{4} & \tfrac{1}{4} & -\tfrac{1}{4} \\
-\tfrac{1}{4} & \tfrac{1}{4} & -\tfrac{1}{4} & \tfrac{1}{4}
\end{matrix}$ \\ \hline
$i=41$
& $\begin{matrix}
-1 & 1 & 1 & 1 \\
-1 & 1 & 1 & 1 \\
-1 & 1 & 1 & 1 \\
-1 & 1 & 1 & 1
\end{matrix}$
& $\begin{matrix}
0 & 4t & 1 & 0 \\
0 & 0 & 0 & 0 \\
0 & -4t & -1 & 0 \\
0 & 0 & 0 & 0
\end{matrix}$
& $\begin{matrix}
\tfrac{1}{4} & 0 & \tfrac{1}{4} & 0 \\
-\tfrac{1}{4} & 0 & -\tfrac{1}{4} & 0 \\
\tfrac{4t-1}{4} & 0 & \tfrac{4t-1}{4} & 0 \\
\tfrac{1}{4} & 0 & \tfrac{1}{4} & 0
\end{matrix}$ \\ \hline
$i=42$
& $\begin{matrix}
0 & 0 & 0 & 1 \\
0 & 0 & 0 & -1 \\
0 & 0 & 0 & 1 \\
0 & 0 & 0 & -1
\end{matrix}$
& $\begin{matrix}
0 & 0 & 0 & 0 \\
0 & 0 & 0 & 0 \\
1 & \tfrac{4t-1}{4t} & 0 & -\tfrac{1}{4t} \\
1 & \tfrac{4t-1}{4t} & 0 & -\tfrac{1}{4t}
\end{matrix}$
& $\begin{matrix}
0 & 0 & 0 & 0 \\
0 & 0 & \tfrac{1}{2} & -\tfrac{1}{2} \\
0 & 0 & 0 & 0 \\
0 & 0 & -\tfrac{1}{2} & \tfrac{1}{2}
\end{matrix}$ \\ \hline
$i=43$
& $\begin{matrix}
1 & 0 & 0 & 0 \\
1 & 0 & 0 & 0 \\
-1 & 0 & 0 & 0 \\
1 & 0 & 0 & 0
\end{matrix}$
& $\begin{matrix}
1 & 0 & 0 & 0 \\
0 & 0 & 0 & 0 \\
1 & 0 & 0 & 0 \\
0 & 0 & 0 & 0
\end{matrix}$
& $\begin{matrix}
0 & \tfrac{1}{2} & 0 & \tfrac{1}{2} \\
0 & 0 & 0 & 0 \\
0 & 0 & 0 & 0 \\
0 & 2t & 0 & 2t
\end{matrix}$ \\ \hline
$i=44$
& $\begin{matrix}
1 & 0 & 1 & 0 \\
1 & 0 & -1 & 0 \\
1 & 0 & 1 & 0 \\
-1 & 0 & 1 & 0
\end{matrix}$
& $\begin{matrix}
4t & 0 & -1 & -1 \\
0 & 0 & 0 & 0 \\
1-4t & 1 & 1 & 1 \\
0 & 0 & 0 & 0
\end{matrix}$
& $\begin{matrix}
\tfrac{1}{4} & -\tfrac{1}{4} & \tfrac{1}{4} & \tfrac{1}{4} \\
\tfrac{1}{4} & \tfrac{1}{4} & \tfrac{1}{4} & -\tfrac{1}{4} \\
-t & -t & -t & t \\
\tfrac{8t-1}{4} & -\tfrac{1}{4} & \tfrac{8t-1}{4} & \tfrac{1}{4}
\end{matrix}$ \\ \hline
$i=45$
& $\begin{matrix}
1 & 1 & -1 & 1 \\
1 & 1 & -1 & 1 \\
-1 & -1 & 1 & -1 \\
1 & 1 & -1 & 1
\end{matrix}$
& $\begin{matrix}
0 & 0 & 0 & 0 \\
0 & 0 & 0 & 0 \\
-1 & 0 & 0 & 0 \\
1 & 0 & 0 & 0
\end{matrix}$
& $\begin{matrix}
\tfrac{1}{4} & \tfrac{1}{4} & 0 & 0 \\
-\tfrac{1}{4} & -\tfrac{1}{4} & 0 & 0 \\
t & t & 0 & 0 \\
0 & 0 & 0 & 0
\end{matrix}$ \\ \hline
$i=46$
& $\begin{matrix}
1 & 1 & 1 & -1 \\
1 & 1 & 1 & -1 \\
1 & 1 & 1 & -1 \\
-1 & -1 & -1 & 1
\end{matrix}$
& $\begin{matrix}
0 & 0 & 0 & 0 \\
0 & 0 & 0 & 0 \\
1-4t & 1 & 1 & 1 \\
4t-1 & -1 & -1 & -1
\end{matrix}$
& $\begin{matrix}
0 & 0 & 0 & 0 \\
0 & 0 & 0 & 0 \\
0 & 0 & 0 & 0 \\
\tfrac{1}{4} & \tfrac{1}{4} & 0 & 0
\end{matrix}$ \\ \hline
$i=47$
& $\begin{matrix}
-1 & 0 & 0 & 0 \\
-1 & 0 & 0 & 0 \\
1 & 0 & 0 & 0 \\
1 & 0 & 0 & 0
\end{matrix}$
& $\begin{matrix}
1 & 0 & -\tfrac{1}{4t} & 0 \\
1 & 0 & -\tfrac{1}{4t} & 0 \\
0 & 0 & 0 & 0 \\
0 & 0 & 0 & 0
\end{matrix}$
& $\begin{matrix}
0 & 0 & \tfrac{1}{2} & \tfrac{1}{2} \\
0 & 0 & 0 & 0 \\
0 & 0 & 0 & 0 \\
0 & 0 & 2t & 2t
\end{matrix}$ \\ \hline
$i=48$
& $\begin{matrix}
0 & -1 & 0 & -1 \\
0 & -1 & 0 & 1 \\
0 & 1 & 0 & 1 \\
0 & -1 & 0 & 1
\end{matrix}$
& $\begin{matrix}
0 & 0 & 0 & 0 \\
0 & 1 & 0 & 0 \\
0 & 0 & 0 & 0 \\
1 & 0 & 0 & 0
\end{matrix}$
& $\begin{matrix}
-\tfrac{1}{4} & \tfrac{1}{4} & \tfrac{1}{4} & \tfrac{1}{4} \\
-\tfrac{1}{4} & -\tfrac{1}{4} & \tfrac{1}{4} & -\tfrac{1}{4} \\
-t & t & t & t \\
\tfrac{1}{4} & \tfrac{1}{4} & -\tfrac{1}{4} & \tfrac{1}{4}
\end{matrix}$ \\ \hline
\end{longtable}

\endgroup

\end{document}